\documentclass{amsart}

\usepackage{verbatim}
\usepackage{graphicx}
\usepackage{amsfonts,amsmath,latexsym,amssymb,amsthm,stmaryrd,mathtools}
\usepackage{geometry}
\usepackage{xcolor}
\newcounter{theoremcounter}

\newcounter{dummycounter}

\newcounter{quescounter}

\newcounter{emptycounter}
\newcounter{defcounter}

\newtheorem{theorem}[theoremcounter]{Theorem}

\newtheorem{question}[quescounter]{Question}

\newtheorem{lemma}[theoremcounter]{Lemma}

\newtheorem{proposition}[theoremcounter]{Proposition}
\newtheorem{corollary}[theoremcounter]{Corollary}
\newtheorem{remark}{Remark}
\newtheorem{definition}[defcounter]{Definition}

\newcounter{eqncounter}

\numberwithin{equation}{eqncounter}

\def\IR{\mathbb R}
\def\IC{\mathbb C}
\def\IZ{\mathbb Z}
\def\IN{\mathbb N}

\def\IQ{\mathbb Q}

\def\IT{\mathbb T}

\def\t{T}

\def\Q{B}

\renewcommand{\vec}[1]{\mbox{\boldmath$#1$}}
\def\N{\mathcal{N}}

\def\Oseen{{\mathcal{O}}}

\def\Qbar{\overline{\IQ}}

\def\valpha{\vec{\alpha}}
\def\vbeta{\vec{\beta}}
\def\vxi{\vec{\xi}}

\def\vtau{\vec{\tau}}

\def\vcxii{\vec{c}_{\xi_i}}
\def\vcxi{\vec{c}_{\xi}}
\def\vcgammai{\vec{c}_{\gamma_i}}
\def\vcgammainull{\vec{c}_{\gamma_{i+i_0}}}
\def\u{u}

\def\xij{\xi_{j}^{(i)}}

\def\House#1{{%
    \setbox0=\hbox{$#1$}
    \vrule height \dimexpr\ht0+1.4pt width .4pt depth \dp0\relax
    \vrule height \dimexpr\ht0+1.4pt width \dimexpr\wd0+2pt depth \dimexpr-\ht0-1pt\relax
    \llap{$#1$\kern1pt}
    \vrule height \dimexpr\ht0+1.4pt width .4pt depth \dp0\relax
}}

\def\house#1{{%
    \setbox0=\hbox{$#1$}
    \vrule height \dimexpr\ht0+1.4pt width .4pt depth \dp0\relax
    \vrule height \dimexpr\ht0+1.4pt width \dimexpr\wd0+2pt depth \dimexpr-\ht0-1pt\relax
    \llap{$#1$\kern1pt}
    \vrule height \dimexpr\ht0+1.4pt width .4pt depth \dp0\relax
}}

\def\housealp{{%
    \setbox2=\hbox{$\alpha$}
    \vrule height \dimexpr\ht2+1.75pt width .4pt depth \dp2\relax
    \vrule height 6.05pt width \dimexpr\wd2+2pt depth -5.65pt
    \llap{$\alpha$\kern1pt}
    \vrule height \dimexpr\ht2+1.75pt width .4pt depth \dp2\relax
}}

\def\Hom{\mathrm{Hom}}

\title{Lehmer without Bogomolov}
\address{Fabien Pazuki. University of Copenhagen, Institute of Mathematics, Universitetsparken 5, 2100 Copenhagen, Denmark, and Universit\'e de Bordeaux, IMB, 351, cours de la Lib\'eration, 33400 Talence, France.}
\email{fpazuki@math.ku.dk}
\address{Niclas Technau. Department of Mathematics, University of Wisconsin-Madison, 480 Lincoln Dr, Madison, WI-53706, USA}
\email{technau@wisc.edu}
\address{Martin Widmer. Department of Mathematics, Royal Holloway, University of London,
Egham, Surrey, TW20 0EX, United Kingdom}
\email{martin.widmer@rhul.ac.uk}

\title{Northcott numbers for the house and the Weil height}
\author{Fabien Pazuki, Niclas Technau and Martin Widmer}

\dedicatory{In memory of  Professor Andrzej Schinzel, 1937--2021.}


\begin{document}

\maketitle

\begin{abstract}
For an algebraic number $\alpha$ and $\gamma\in \IR$, let $\housealp$ be the house, $h(\alpha)$ be the (logarithmic) Weil height,
and $h_\gamma(\alpha)=(\deg\alpha)^\gamma h(\alpha)$ be the $\gamma$-weighted (logarithmic) Weil height of $\alpha$.
Let $f:\Qbar\to [0,\infty)$ be a function on the algebraic numbers $\Qbar$, and let  $S\subset \Qbar$.
The Northcott number $\N_f(S)$ of  $S$, with respect to $f$,
is the infimum of all $X\geq 0$ such that 
$\{\alpha \in S; f(\alpha)< X\}$ is infinite. 
This paper studies the set of Northcott numbers $\N_f(\Oseen)$ for subrings of $\Qbar$
for the house, the Weil height,
and  the $\gamma$-weighted Weil height.
We show:
\begin{enumerate}
\item Every $t\geq 1$ is the Northcott number of a ring of integers of a field  w.r.t. the house $\house{\cdot}$.
\item For each $t\geq 0$ there exists a field with Northcott number in $ [t,2t]$ w.r.t. the Weil height $h(\cdot)$.
\item For all $0\leq \gamma\leq 1$ and $\gamma'<\gamma$ there exists a field $K$ with $\N_{h_{\gamma'}}(K)=0$ and $\N_{h_\gamma}(K)=\infty$.
\end{enumerate}
For $(1)$ we provide examples that satisfy an analogue of Julia Robinon's property (JR),  examples that satisfy an analogue of
Vidaux and Videla's isolation property, and examples that satisfy  neither of those. 
Item $(2)$ concerns a question raised by Vidaux and Videla due to
its direct link with decidability theory via the Julia Robinson number.
Item (3) is a strong generalisation of the known fact that there are fields that satisfy the Lehmer conjecture but 
which are not Bogomolov in the sense of Bombieri and Zannier.
\end{abstract}

\section{Introduction}
In this article we investigate the spectrum of Northcott numbers 
of subrings of the algebraic numbers $\Qbar$ for the house and the Weil height.
The Northcott number with respect to the Weil height was introduced 
by Vidaux and Videla \cite{ViVi16}, and refines the concept of the Northcott property 
which goes back to Northcott \cite{Northcott49, Northcott50}
but was formally defined by Bombieri and Zannier \cite{BoZa}.  
Northcott numbers for various other height functions have been around implicitly 
and explicitly in the Bogomolov property, the
Lehmer conjecture, the Schinzel--Zassenhaus conjecture (now Dimitrov's theorem), 
the Julia Robinson property, and  the Julia Robinson number. 
To unify all these concepts under the umbrella of Northcott numbers
 we start with the following obvious generalisation.
\begin{definition}[Northcott number]\label{Northcottnum}
For a subset $S$ of  the algebraic numbers $\Qbar$ and $f:\Qbar\to [0,\infty)$
we set 
$$\N_f(S)=\inf\{t\in [0,\infty); \#\{\alpha\in S; f(\alpha)< t\}=\infty\},$$
with  the usual interpretation $\inf \emptyset=\infty$. 
We call $\N_f(S)\in [0,\infty]$ the \emph{Northcott number} of $S$ (with respect to $f$).
If $\N_f(S)=\infty$ then we say that $S$ has the Northcott property (with respect to $f$). 
\end{definition}

Throughout this introduction ring always means not the zero ring.
Next we give some background on the relevant results that use the
 house $\house{\cdot}$ of an algebraic number 
(i.e., the maximum modulus of its conjugates over $\IQ$).

In 1959, Julia Robinson \cite{Robinson1959}  showed the undecidability of the first order theory of any number field, extending the case $\IQ$ dealt with in her  Ph.D. dissertation. 
A few years later she began \cite{Robinson1962, JuliaRobinson42} to investigate 
decidability questions for certain rings of totally real algebraic 
integers of infinite degree. To this end she introduced the following 
property, nowadays called property (JR).
Let $\Oseen$ be a ring of totally real algebraic integers, 
and let $\Oseen^+\subset \Oseen$ be its subset of totally positive elements.
The ring $\Oseen$ has property (JR) if  the following  holds
$$ \#\{\alpha\in \Oseen^+; \housealp< \N_{\house{\cdot}}(\Oseen^+)\}=\infty.$$
As usual, $x< \infty$ is true for all $x\in \IR$ by convention. 
Note that the Northcott property implies the property (JR).
Let $\IN=\{1,2,3,\ldots\}$ be the set of positive integers, and set  $\IN_{0}=\IN \cup \{0\}$.

Robinson showed that the semi-ring $(\IN_{0},0,1,+,\cdot)$ is first order definable in $\Oseen$
for any ring $\Oseen$ of totally real algebraic integers with property (JR)
(not necessarily the ring of integers of a field, as pointed out by Vidaux and Videla \cite{ViVi15AMS}).
She then proved that the rings of integers $\Oseen_K$ of the maximal 
totally real extension $K$ of $\IQ$, and  of  $K=\IQ(\sqrt{n}; n\in \IN)$
both have property (JR): 
the former since the infimum $4$ in the definition of 
$\N_{\house{\cdot}}(\Oseen_K^+)$ is attained, 
and the latter since it has the Northcott property. 
Hence, both have undecidable first order theory.
Since the aforementioned field is a pro-$2$ extension of $\IQ$, it follows from  a result of Videla \cite{videla2000definability},
that its  ring of integers  is  first order definable in this field, and thus  the field inherits the undecidability from its ring of integers.  

A question that arose from Robinson's work, explicitly proposed by Vidaux and Videla \cite[Question 1.5]{ViVi15AMS}, 
is, which numbers can be realised as Northcott numbers\footnote{Vidaux and Videla call $\N_{\House{\cdot}}(\Oseen^+)$ the Julia Robinson number of the ring $\Oseen$.} $\N_{\house{\cdot}}(\Oseen_K^+)$.
Important progress on this question was made by 
Gillibert and Ranieri \cite{GillibertRanieri} who proved that all numbers of the form $[2\sqrt{2n}]+2\sqrt{2n}$ or $8n$, with $n\geq 1$ odd and square-free, are of this type.
Further results on the distribution of the Northcott numbers $\N_{\House{\cdot}}(\Oseen^+)$ were obtained by Castillo \cite{CastilloAA}, and Castillo, Vidaux, and Videla \cite{CVV20}.  

Another question, explicitly proposed by Robinson herself, is, if in fact the  ring of integers $\Oseen_K$ of every totally real field $K$ has property (JR).
Gillibert and Ranieri \cite{GillibertRanieri} noted that all their examples do have property (JR).

Vidaux and Videla \cite[Definition 1.2]{ViVi15AMS}  introduced a related condition which they call isolation property, and which also allows, by the same strategy as for the property (JR),
to define the semi-ring $(\IN_0,0,1,+,\cdot)$ by a  first order formula in $\Oseen$. 
A ring $\Oseen$  of totally real algebraic integers has the isolation property if it does not have property (JR), and if 
there exists $M>\N_{\house{\cdot}}(\Oseen^+)$ such that for all $\epsilon>0$ we have 
$$\#\{\alpha\in \Oseen^+; \N_{\house{\cdot}}(\Oseen^+)+\epsilon\leq \housealp<M\}<\infty.$$
Since there are only  finitely many totally real integers that assume a fixed house value $t$ (in particular that assume the value $\N_{\house{\cdot}}(\Oseen^+)$) it follows that the above cardinality gets arbitrarily large as $\epsilon$ gets small.
Vidaux and Videla \cite{ViVi15AMS} have constructed rings of totally real algebraic integers that satisfy their isolation property but it is unknown if any of these is the ring of integers $\Oseen_K$
of a field, and so Robinson's question also remains open.
Nevertheless, examining decidability of subrings and subfields of $\Qbar$
by Julia Robinson's strategy (and refinements thereof)
is an active area of research. From the growing body of literature, 
we refer the reader to the work (and references therein) 
of Shlapentokh \cite{Shlapentokh2018},
Springer \cite{Springer2020}, 
as well as Mart{\'\i}nez-Ranero, Utreras, and Videla \cite{Martinez2020}. \\

Our first result shows that if we consider the full set of algebraic integers, 
and we do not restrict to totally real fields, then every real number $t\geq 1$ 
is a Northcott number with respect to the house.
Furthermore, the analogous question to Julia Robinson's one can be answered in the negative, i.e., the infimum in the definition of $\N_{\house{\cdot}}(\Oseen_{K})$ is not 
always attained. Finally, we can also construct rings of integers $\Oseen_K$ with given Northcott number  that neither have the analogue of property (JR) nor the 
analogue of the isolation property\footnote{It seems natural to 
impose the additional condition $\N_{\house{\cdot}}(\Oseen)$ is attained only for finitely many elements of $\Oseen$ for the analogue of the isolation property in the non totally real case,
since this condition  automatically holds only in the totally real case.}.  

 \begin{theorem}\label{thm: spectrum full for house function}
Let $t>1$ be a real number. 
\begin{enumerate}
\item[(a)]There exists a field $K$ of algebraic numbers such that its ring
of integers $\Oseen_{K}$ satisfies $\N_{\house{\cdot}}(\Oseen_{K})=t$ and $\#\{\alpha\in \Oseen_K; \housealp< t\}=\infty$.

\item[(b)]There exists $M>t$ and a field $K$ of algebraic numbers such that its ring
of integers $\Oseen_{K}$ satisfies $\N_{\house{\cdot}}(\Oseen_{K})=t$ and $\#\{\alpha\in \Oseen_K; \housealp\leq t \text{ or } t+\epsilon\leq \housealp< M\}<\infty$ for all $\epsilon>0$.

\item[(c)]There exists a field $K$ of algebraic numbers such that its ring
of integers $\Oseen_{K}$ satisfies $\N_{\house{\cdot}}(\Oseen_{K})=t$ and $\#\{\alpha\in \Oseen_K; \housealp\leq t\}<\infty$ and $\#\{\alpha\in \Oseen_K; t+\epsilon\leq \housealp< M\}=\infty$  for all $M>t$ and all small enough $\epsilon>0$.
\end{enumerate}
\end{theorem}

Since $\housealp\geq 1$ for every non-zero algebraic integer  
there is no ring of algebraic integers $\Oseen$ for which 
$\#\{\alpha\in \Oseen; \housealp< 1\}=\infty$.
But, by our method, it is easy to construct fields $K$ 
whose ring of integers have Northcott number $t=1$, 
and that satisfy either selection of the  remaining two  properties.

The proof of  Theorem \ref{thm: spectrum full for house function} comes in two steps.
First we construct a ring with prescribed Northcott number (and the additional topological features),
and then we prove that the constructed ring is integrally closed (in its field of fractions).
For the latter we  exploit a criterion of Dedekind, demanding our construction to satisfy
certain congruence constraints. The Siegel--Walfisz theorem about the distribution of primes in residue
classes ensures that we can satisfy these congruence conditions.  

The original problems considered by Robinson, and by Vidaux and Videla (restricting to $\Oseen_K^+$ for totally real fields $K$) are more difficult
than those we address in Theorem \ref{thm: spectrum full for house function}. However, it is conceivable that the methods in this paper  are also useful to address these original questions.\\

Our construction of rings with prescribed Northcott number relies on our next result. 
Consider a sequence $(\xi_i)_i$ of  algebraic integers, let $\Oseen_0$ be a ring, containing $1$, of algebraic integers,
$\Oseen_i=\Oseen_0[\xi_1,\ldots,\xi_i],$
and let $\Oseen=\bigcup_{i\geq 1} \Oseen_i=\Oseen_0[\xi_1,\xi_2,\xi_3,\ldots].$
Let $K_i$ be the field of fractions of $\Oseen_i$, and set 
$d_i=[K_{i-1}(\xi_i):K_{i-1}].$
For a subfield $K\subset \Qbar$ and an algebraic number $\xi$ let $M_{\xi,K}\in K[x]$ be the monic minimal polynomial of 
$\xi$ over $K$. We introduce a new quantity $\eta(K,\xi)$ which measures the 
largest root of $\sigma(M_{\xi,K})$ and how equidistributed the normalised roots\footnote{We say $d$ complex points are (perfectly) equidistributed on a circle (of radius $R$) if they are pairwise distinct, all lie on the circle, and the arc-length between neighboring points is $2\pi R/d$. By ``normalised'' we mean scaled by the reciprocal of the largest of their moduli.}  of $\sigma(M_{\xi,K})$ on the unit circle are
for each field homomorphism $\sigma:K\to \IC$. The definition 
of $\eta(K,\xi)$ is given in Section \ref{sec:Galoisorbits}.
We always consider $\liminf$ as element of the extended real number line $\IR\cup\{\pm \infty\}$.

\begin{theorem}\label{thm:Nnumbergeneral}
Suppose that $\N_{\house{\cdot}}(\Oseen_i)=\infty$, $d_i>1$, and that $M_{\xi_i,K_{i-1}}\in \Oseen_{i-1}[x]$ for all  $i\in \IN$.
Then
$$\N_{\house{\cdot}}(\Oseen)\geq \liminf_{i\rightarrow \infty} \eta(K_{i-1},\xi_i).$$
\end{theorem}
Since $d_i>1$ the $\xi_i$ are pairwise distinct, and thus we also have the trivial upper bound
$$\N_{\house{\cdot}}(\Oseen)\leq \liminf_{i \rightarrow \infty} \house{\xi_i}.$$

The simplest application of  Theorem \ref{thm:Nnumbergeneral} is when  $M_{\xi_i,K_{i-1}}=M_{\xi_i,\IQ}$
and the conjugates over $\IQ$ are perfectly equidistributed on a circle $|z|=t_i$,
e.g., if they are of the form $\xi_j^{(i)}=t_i\zeta_{d_i}^j$ $(1\leq j\leq d_i)$.  In Section \ref{sec:applications} we explain this and other  applications,
including a more  sophisticated result (Corollary \ref{cor:Nnumber}), that requires the full strength of  Theorem \ref{thm:Nnumbergeneral}.\\

The Northcott number $\N_{\house{\cdot}}(\Oseen_K)$ is also related to the invariant $c_1(K)$ 
for fields $K\subset \Qbar$ introduced by Gaudron and R\'emond in their investigations of the Siegel property for fields.
This invariant is often difficult to determine; however, 
they show \cite[Lemme 5.4]{GaudronRemondSiegel} that
 $c_1(K)\geq \N_{\house{\cdot}}(\Oseen_K)$ provided $K$ has infinite degree over $\IQ$.
 
They also provide an example \cite[Exemple 4.6]{GaudronRemondSiegel} of a field $K$ 
that has infinitely many elements of bounded Weil height but whose ring of integers has only finitely many elements of bounded house, i.e., 
$\N_h(K)<\infty =\N_{\house{\cdot}}(\Oseen_K)$.
Their proof of $\N_{\house{\cdot}}(\Oseen_K)=\infty$  relies on the (perfect) orthogonality relations of the roots of unity, and could be adapted to 
handle the aforementioned simplest  case $\xij=t_i\zeta_{d_i}^j$.
Their method, has the advantage that it can deal with integral elements in  $\IQ[\xi_1,\xi_2,\xi_3,\ldots]$ but, in contrast to ours, it seems restricted to the
perfectly equidistributed case, and cannot provide results such as  Corollary \ref{cor:Nnumber} of Section \ref{sec:applications}.\\

The next height function we consider is the classical logarithmic absolute Weil height $h(\cdot)$.   
Again, we first give some background, and then we state our result.

The origin of the Northcott property 
goes back to two seminal papers of D.G. Northcott \cite{Northcott49, Northcott50} from  1949 and 1950, in which he showed  that there are only finitely many algebraic 
numbers of bounded degree and bounded Weil height $h(\cdot)$, and proved the finiteness of the number of preperiodic points of bounded degree 
under non-linear algebraic endomorphisms  of projective varieties defined over $\Qbar$.  

The Northcott property (with respect to $h(\cdot)$) is well known to have many diophantine applications, and thus it is natural to refine this concept via  the Northcott number
as done by Vidaux and Videla \cite{ViVi16}.
Indeed, it is often enough to know that the Northott number of a specified set is a sufficiently large finite number.
For instance, to show that the non-linear polynomial $f\in K[x]$ has only finitely many preperiodic points in the field $K\subset \Qbar$ it suffices to know that $\N_h(K)>2c_f$ where $h(f(\alpha))\geq  \deg f\cdot h(\alpha)-c_f$. 
Even more concretely, for  the polynomial $f_n=x^{2n}-x^{2n-1}+\cdots -x +1$ one can take\footnote{Note that with $g_1=y^n$ and $g_2=y^nf_n(x/y)$ we have $x^{n+1}=-yg_1+(x+y)g_2$ and $y^{n+1}=yg_1$. From this it is routine to compute $c_{f_n}$.}  $c_{f_n}=2\log 2$.

On the  opposite end, the first and the last author  \cite{PazukiWidmer1} have recently proved an arithmetic Bertini-type result for which 
fields with prescribed arithmetic features and sufficiently small Northcott number are needed.

These observations  raise the question which numbers can be realised as the Northcott number of a field or a ring of integers of a subfield of $\Qbar$.
A similar question was  raised by Vidaux and Videla \cite[Question 6]{ViVi16}. 
\begin{question}[Vidaux, Videla 2016]\label{QuestionViVi}
Which real numbers can be realised as Northcott number (with respect to  the absolute logarithmic Weil height) of a ring extension of $\IQ$?
\end{question}
Interestingly, Vidaux and Videla's motivation for the above Question \ref{QuestionViVi} comes from their earlier question about the spectrum of the 
Julia Robinson numbers (i.e., the spectrum of the Northcott numbers $\N_{\house{\cdot}}(\Oseen^+_K)$  for totally real fields $K$), and the fact that $\housealp\geq h(\alpha)$ for every non-zero algebraic integer.  Given their motivation it seems to us equally natural to propose the analogous question for the house $\house{\cdot}$ --- a question  that is completely answered by Theorem \ref{thm: spectrum full for house function}.

However, back to the Weil height $h(\cdot)$. To the best of our knowledge,
 there are currently only two possible ``values'' known as Northcott numbers for subrings of 
$\Qbar$, namely $0$ (attained, e.g., by $\Qbar$) and $\infty$
(attained, e.g., by any number field).  
Here we  show that the set of values cannot be sparse. 
\begin{theorem}\label{Prop2}
Let $t\geq 0$. There exists a field $L\subset \Qbar$ satisfying 
$$
t\leq \N_{h}(L)\leq \N_{h}(\Oseen_L)\leq 2t.$$
More precisely,
every field $L$  generated over  a number field $K$ 
by any sequence of roots $p_i^{1/d_i}$ that converge to $\exp(2t)$, and 
where $p_i$ and $d_i$ are primes and the $p_i$ are strictly increasing, satisfies the conclusion.
\end{theorem}

For the aforementioned example $f_n=x^{2n}-x^{2n-1}+\cdots -x +1\in \IZ[x]$,
we conclude from Theorem \ref{Prop2} that $f_n$ has only finitely many preperiodic points in  $L$ (with a bound independent of $n$), provided $t>4\log 2$.

Finally, let us mention that Gaudron and R\'emond's  Siegel property \cite{GaudronRemondSiegel} is also related to $\N_{h}(K)$.
For instance, they show \cite[Corollaire 1.2]{GaudronRemondSiegel} that  if $K$ is a Siegel field of  infinite degree over $\IQ$ then $\N_{h}(K)<\infty$.
\\

Our last result is concerned with Northcott numbers for differently normalised Weil heights. Many results around the Lehmer conjecture can be expressed 
in terms of the Northcott number of a suitably normalised Weil height. For example, 
writing $\mu\subset \Qbar$ for the set of roots of unity and $\deg \alpha=[\IQ(\alpha):\IQ]$, Dobrowolski's Theorem states that  $\N_f(\Qbar\backslash\mu)>0$ for\footnote{For $x\in \IR$ we write $\log^+x=\log\max\{x,\exp(1)\}$.} $f(\alpha)=\left(\frac{\log^+ \deg \alpha}{\log^+\log \deg \alpha}\right)^3(\deg \alpha)h(\alpha)$. Let us now restrict ourselves to the case where  
$$f(\alpha)=h_\gamma(\alpha)=(\deg \alpha)^\gamma h(\alpha)\quad \text{ 
for some } \gamma\in \IR.$$

Lehmer's conjecture itself states that $\N_{h_1}(\Qbar\backslash\mu)>0$. And the Bogomolov property for a set $S\subset \Qbar$, also introduced by Bombieri and Zannier \cite{BoZa}, can be rephrased as 
$\N_{h_0}(S\backslash\mu)>0$. In analogy the first author and Pengo \cite{PazukiPengo20}
say the set $S$ has the Lehmer property if  
$\N_{h_1}(S\backslash\mu)>0$. Generalising both properties we  say a set $S\subset \Qbar$ is  $\gamma$-Bogomolov
if $\N_{h_\gamma}(S\backslash\mu)>0$, and we say $S$ is  $\gamma$-Northcott
if $\N_{h_\gamma}(S)=\infty$ (i.e., $S$ has the Northcott property with respect to $h_{\gamma}(\cdot)$).
Note that by Dobrowolski's Theorem the field $\Qbar$ (and hence each of its subsets) is $\gamma$-Bogomolov  for every $\gamma>1$.

Amoroso's Theorem 1.3 in \cite{Amoroso2016} shows that the field $\mathbb{Q}(\zeta_3, 2^{1/3}, \zeta_{3^2}, 2^{1/{3^2}}, \zeta_{3^3}, 2^{1/{3^3}}, \ldots)$, where $\zeta_d$ denotes a primitive $d$-th root of unity, is  $1$-Bogomolov but not $0$-Bogomolov. 
Another example, as we explain now, of such a field
is $\IQ^{tr}(\sqrt{-1})$, where $\IQ^{tr}$ 
denotes the maximal totally real extension of $\mathbb{Q}$. 
By a result of Schinzel \cite[Theorem 2]{Schinzel1973} we have $(\deg\alpha) h(\alpha)\geq \log(\frac{1+\sqrt{17}}{4})/2$ for every unit $\alpha$ in the ring of integers of $\IQ^{tr}(\sqrt{-1})$.
This implies that $\IQ^{tr}(\sqrt{-1})$ is $1$-Bogomolov.
But from Example 5.3 in \cite{AmorosoDavidZannier}, the sequence 
$$
\bigg(\Big(\frac{2+\sqrt{-1}}{2-\sqrt{-1}}\Big)^{1/k}\bigg)_{k\geq1}
$$ has all its elements in $\IQ^{tr}(\sqrt{-1})$ which shows that this field is not $0$-Bogomolov.

This raises the question whether for every $\gamma\leq 1$ and $\epsilon>0$ there exists a field $K$ that is $\gamma$-Bogomolov (or even $\gamma$-Northcott)  but not 
$(\gamma-\epsilon)$-Bogomolov.
Our next result answers this question in the affirmative.
\begin{theorem}\label{PropNBGexplicit}
Let $0\leq \gamma\leq 1$, and $\epsilon>0$. Choose sequences of primes
 $(d_i)_{i\in \IN}$ and $(p_i)_{i\in \IN}$ such that $d_{i+1}\geq 2d_{i}$, and 
$d_i^{1-\gamma+\epsilon/2}\leq \log p_i\leq \log (2)+ d_i^{1-\gamma+\epsilon/2}$ for all $i\in \IN$.
Then $\IQ(p_i^{1/d_i}; i\in \IN)$ is  $\gamma$-Northcott but not $(\gamma-\epsilon)$-Bogomolov. 
\end{theorem}
While the proofs of Theorem \ref{Prop2} and of Theorem \ref{PropNBGexplicit} rely on a method from \cite{WidmerPropN}, 
the proof of Theorem \ref{thm: spectrum full for house function} is essentially different and is based on an equidistribution argument.
However, it turns out that both methods are particularly easy to apply for fields of the shape $\IQ(p_i^{1/d_i}; i\in \IN)$
for certain primes $p_i$ and $d_i$, and this is the reason that all fields constructed in these three theorems are of this type.

\section*{acknowledgements}
We thank Lukas Pottmeyer for pointing out Theorem 1.3 in \cite{Amoroso2016} 
and  $\IQ^{tr}(\sqrt{-1})$
as examples of fields that are $1$-Bogomolov but not $0$-Bogomolov. We also thank
Xavier Vidaux and Carlos Videla for useful feedback on an earlier version,
and for alerting us to the work of M. Castillo.
Finally, we thank the anonymous referee for carefully reading our manuscript, and for various helpful suggestions that improved the exposition of our paper.
NT is supported by the Austrian Science Fund (FWF): project J 4464-N. 
FP is supported by ANR-17-CE40-0012 Flair and ANR-20-CE40-0003 Jinvariant.

\section{Definitions and basic properties of Northcott numbers}\label{sec:geometriclowerbound}

In this section, we introduce some notation, and collect some basic results about Northcott numbers.
 
We write $|\cdot|$ for the usual absolute value on $\IC$, and for the maximum norm on $\IC^d$ we set
$$\|\valpha\|=\max_{1\leq i\leq d} |\alpha_{i}|.$$ 
For a field $K$
of characteristic $0$ we denote the set of field homomorphisms $\sigma:K\to \IC$ by
$$\mathrm{Hom}(K)=\{\sigma:K\to \IC; \text{ field homomorphism}\}.$$
The house of an algebraic number $\alpha$, written $\housealp$, is defined as 
$$\housealp=\max_{\sigma \in \Hom(\Qbar)}|\sigma(\alpha)|.$$

Next we define the Weil height.
To this end let $K$ be a number field, and let $M_K$ denote the set of places of $K$,
that is, equivalence classes of absolute values.
For each place $v\in M_K$ we let $|\cdot |_v$ be the unique representative of $v$ 
that extends one of the canonical absolute values\footnote{I.e., for $x\in \mathbb{Q}$, we have 
$\vert x\vert_v=\max\{-x,x\}$ if $v$
is archimedean, and otherwise $\vert p\vert_v=p^{-1}$ 
if $v$ lies above the prime $p$.} on $\IQ$.
For $v\in M_K$, let $K_v$ abbreviate the completion of $K$ with respect to $|\cdot|_v$.
The (logarithmic) Weil height of $\alpha\in K$ is given by
$$
h(\alpha) = \frac{1}{[K:\mathbb{Q}]} 
\sum_{v\in M_K} [K_v : \mathbb{Q}_v] \log\max\{1, \vert x \vert_v\}.
$$
The product formula implies that $h(\alpha)$ 
does not depend on the ambient field $K$, and hence $h(\cdot)$ extends to a function on $\Qbar$. Further,  
we have $ h(\sigma (\alpha)) = h(\alpha)$ for any field homomorphism $\sigma: \overline{\IQ} \rightarrow \Qbar$,
and $h(\alpha+\beta)\leq h(\alpha)+h(\beta)+\log 2$, and $h(\alpha\beta)  \leq h(\alpha)+ h(\beta)$ for all
$\alpha,\beta \in \Qbar$.
More generally, suppose we have a function $f:\Qbar\to [0,\infty)$ and that there exists a continuous function $F:\IR^2\to [0,\infty)$
such that for any field homomorphism $\sigma: \overline{\IQ} \rightarrow \Qbar$, 
and all algebraic numbers $\alpha,\beta$
the following is true
\begin{align*}
&(f1)  & f(\sigma (\alpha)) &= f(\alpha),\\ \,\,
&(f2)  & f(\alpha+\beta)&\leq F(f(\alpha),f(\beta)),\\ \,\,
&(f3) & f(\alpha\beta) & \leq F(f(\alpha),f(\beta)).\,\,
\end{align*}
With $F(x,y)=\max\{xy,x+y+\log 2\}$ the properties (f1), (f2), and (f3) are satisfied for
the Weil height $h(\cdot)$ and for the house $\house{\cdot}$. 
Furthermore, note that for each non-zero algebraic integer $\alpha$ we have 
\begin{alignat*}1
h(\alpha)\leq \log \housealp.
\end{alignat*}
Dvornicich and Zannier observed that the proof of Northcott's Theorem yields a more general statement, which we 
state here in an even slightly more general form.
\begin{lemma}[Dvornicich and Zannier {\cite[Thm. 2.1]{DvZa}}]\label{lem:generalNorthcott}
Suppose $f$ from Definition \ref{Northcottnum} satisfies $(f1)$, $(f2)$, and $(f3)$.
Let $K$ be a subfield of $\Qbar$,
and $U\subset K$. Let $S\subset \Qbar$ 
be a set of roots of monic irreducible polynomials in $K[x]$
with coefficients in $U$ and uniformly bounded degree.
If $U$ has the Northcott property with respect to $f$,
then $S$ has the Northcott property with respect to $f$ as well.

\end{lemma}
\begin{proof}
The following is a straightforward adaptation of the proof of \cite[Thm. 2.1]{DvZa}.
For the sake of completeness, we provide the details.
Let $X>0$, and $\alpha \in S$ an element with $f(\alpha) <X$.
If $\beta$ is a conjugate of $\alpha$ over $K$, then $(f1)$ implies
$f(\beta) =f(\alpha) \leq X$. Let $E$ be an integer such that $ [K(\gamma): K]\leq E $
for any element $\gamma \in S$. Denote the monic minimal polynomial of $\alpha$ over $K$ by
$M_{\alpha,K}(x)=a_0+ a_1 x+ \cdots + a_{d-1} x^{d-1}+ x^d $. By assumption, 
$a_i\in U$ for any $0\leq i <d$. 
Next we will exploit that each $a_i$ is an elementary symmetric function
in the conjugates of $\alpha$ (over $K$).
To this end, we first observe that there are at most $d\leq E$ conjugates of $\alpha$ over $K$.
By using the properties $(f2)$ and $(f3)$, and the fact that a continuous function attains its maximum on a compact set,
we infer that $f(a_i)$ is bounded from above in terms of $X,E$ and the function $F(\cdot,\cdot)$ for all $0\leq i<d$. 
Since $U$ has the Northcott property with respect to $f$, 
there are only finitely many such $(a_0, \ldots,a_{d-1})\in U^d$. 
Hence the number of $\alpha \in S$ with $f(\alpha) \leq X$ is finite for any $X>0$.
\end{proof}

The most important case is when $K=\IQ$ and $U=\IQ$ or $U=\IZ$ respectively, from which it follows that each  number field has the Northcott property with respect to $h(\cdot)$, and the ring of integers of each number field has the Northcott property with respect to the house $\house{\cdot}$.
We point out two further immediate consequences of Lemma \ref{lem:generalNorthcott} for  the Weil height and the house.

\begin{remark}\label{rem:sec2}
Suppose that  $K\subset L$ are fields of algebraic numbers and that $[L:K]$ is finite. We have
\begin{itemize}
\item[(a)]  $\N_{h}(K)=\infty$ if and only if $\N_{h}(L)=\infty$,
\item[(b)]   $\N_{\house{\cdot}}(\Oseen_K)=\infty$ if and only if  $\N_{\house{\cdot}}(\Oseen_L)=\infty$.
\end{itemize}
However,  the Northcott number with respect to the Weil height and the house is  not preserved
under finite extensions in general.
Indeed,  $\log \N_{\house{\cdot}}(\Oseen_{\IQ^{tr}})\geq \N_{h}(\IQ^{tr})\geq \frac{1}{2}\log((1+\sqrt{5})/2))$ by Schinzel's result \cite[Theorem 2]{Schinzel1973}
but $\IQ^{tr}(\sqrt{-1})$ contains infinitely many roots 
of unity, and hence, $\log \N_{\house{\cdot}}(\Oseen_{\IQ^{tr}(\sqrt{-1})})= \N_{h}(\IQ^{tr}(\sqrt{-1}))=0$.
\end{remark}

Next we describe a general characterisation of the Northcott number of a set that is represented as a union of an infinite nested sequence of sets.  
Let  $f:\Qbar \to [0,\infty)$.  For each set $S\subset \Qbar$ we set
$$\delta_f(S)=\inf\{f(\alpha); \alpha\in S\}.$$
Let $A_0\subset A_1\subset A_2\subset \cdots$ be a nested sequence of subsets of $\Qbar$, and we set $\displaystyle{A=\cup_{i\geq0} A_i}$.
The next lemma shows that this quantity, capturing the relative behaviour of the height 
function at each step, determines the Northcott number of $A$ under fairly mild assumptions.
\begin{lemma}\label{LemmaN}
Suppose that $\N_f(A_i)=\infty$ for all $i\in\IN_0$. We have\footnote{Of course, if the sequence $(A_i)_i$ becomes stationary, so that $\delta_f(A_{i}\backslash A_{i-1})=\infty$
for all large enough $i$, then the right hand-side is interpreted as $\infty$.}
$$\N_f(A)=\liminf_{i\to \infty} \delta_f(A_{i}\backslash A_{i-1}).$$
\end{lemma}
\begin{proof}
Since  $\N_f(A_i)=\infty$ for all $i\geq0$ we conclude that $\N_f(A)\leq\liminf \delta_f(A_{i}\backslash A_{i-1}).$
To prove that $\N_f(A)\geq \liminf \delta_f(A_{i}\backslash A_{i-1})$ we can assume that $\N_f(A)<\infty.$
There exists  a sequence $(\alpha_i)_i\subset A$ of pairwise distinct elements with $(f(\alpha_i))_i$
converges to $\N_f(A)$. 
For $\alpha_i$ we set $\iota=\iota (\alpha_i)=\min\{l; \alpha_i\in A_l\}$
so  that $\alpha_i\in A_\iota\backslash A_{\iota-1}$. Hence,
\begin{alignat}1\label{proofProp2ineq}
f(\alpha_i)\geq \delta_f(A_{\iota}\backslash A_{\iota-1}).
\end{alignat}
Since  $\N_f(A_i)=\infty$ for all $i$ we infer that $\iota \rightarrow \infty$ as $i\rightarrow \infty$.
As the left hand side in (\ref{proofProp2ineq}) tends to $\N_f(A)$ the claim drops out. 
\end{proof}

\section{A lower bound for the maximum of a unitary complex polynomial on given points}\label{sec:geometriclowerbound}

For $d\in \IN $ let 
\[
\zeta_{d}=\exp\left(\frac{2\pi i}{d}\right)
\]
denote a primitive $d^{th}$-root of unity.
\begin{definition}[Finite discrepancy]\label{def:Discrepancy}
For $\vxi=(\xi_1,\ldots,\xi_d)\in \IC^d$
we set 
\begin{alignat}1
D(\vxi)=\inf_{\phi}\max_{1\leq j\leq d}\min_{1\leq i\leq d}|\xi_i-\phi(\zeta_d^j)|,
\end{alignat}
where $\phi$ runs over all rotations about the origin (i.e., 
$\phi(z)=\u z$ for some $\u\in \IT=\{c\in \IC; \vert c \vert=1\}$).
\end{definition}
Note that $D(\vxi)$  is invariant under permutation of the entries of $\vxi$ - a  fact that will be used in the sequel  
without further notice.
The \emph{multiplicity}, on the other hand,
is important: $D(\xi)=\vert \vert \xi \vert -1\vert$ but $D(\xi,\xi)=\sqrt{1 + \vert \xi \vert^2}$.  
Also, note that $D(\vxi)=0$ if and only if $\xi_1,\ldots,\xi_d$ are perfectly equidistributed on the unit circle, i.e.,
if $\{\xi_1,\ldots,\xi_d\}=\{\phi(\zeta_d),\ldots,\phi(\zeta_d^d)\}$ for some rotation about the origin $\phi$.

\begin{lemma}[$\ell^{2}$-lower bound]
\label{lem: ell two lower bound}
Let $\vxi\in \IC^d$  and suppose $\Q(x)=b_{0}+b_{1}x+\cdots+b_{n}x^{n}\in\mathbb{C}[x]$
has degree strictly less than $d$. We have 
\[
\max_{1\leq i\leq d}\vert \Q(\xi_i)\vert\geq\left(1-d^{3/2}D(\vxi)\max_{1\leq i\leq d}\{1,|\xi_i|\}^{d-2}\right)\sqrt{\sum_{0\leq i\leq n}\vert b_{i}\vert^2}.
\]
If $n=0$ then we can omit the first factor.
\end{lemma}

\begin{proof}
Note that the statement is trivially true for $d=1$. So we can assume $d\geq 2$.
First, let us assume that $\xi_i=\zeta_d^i$ for $1\leq i\leq d$.
We observe that 
\[
\left(\max_{1\leq i\leq d}\vert \Q(\zeta_{d}^{i})\vert\right)^{2}\geq\frac{1}{d}
\sum_{1\leq i\leq d}\vert \Q(\zeta_{d}^{i})\vert^{2}.
\]
Since 
\[
\vert \Q(\zeta_{d}^{i})\vert^{2}=
\sum_{0\leq k,l\leq n}b_{k}\zeta_{d}^{ik}\overline{b_{l}\zeta_{d}^{il}}
=\sum_{0\leq k,l\leq n}b_{k}\overline{b_{l}}\zeta_{d}^{i(k-l)},
\]
we conclude that 
\[
\left(\max_{1\leq i\leq d}\vert \Q(\zeta_{d}^{i})\vert\right)^{2}\geq\frac{1}{d}\sum_{0\leq k,l\leq n}b_{k}\overline{b_{l}}\sum_{1\leq i\leq d}\zeta_{d}^{i(k-l)}.
\]
The inner-most sum vanishes unless $d\mid k-l$.
Because $n<d$ this can only occur if $k=l$, in which case the inner-most
sum equals $d$. Using this and taking the square-root completes the
proof in this case.

Now we note that the same estimate holds true if, $D(\vxi)=0$
i.e., if, after relabeling, $\xi_i=\u \zeta_d^i$ for some fixed $\u$ on the unit circle. Indeed, 
$\Q(\xi_i)=\Q(\u\zeta_d^i)=\tilde{\Q}(\zeta_d^i)$ with $\tilde{\Q}(x)=b_{0}+b_{1}\u x+\cdots+b_{n}\u^n x^{n}$, and the coefficient vectors of 
$\Q$ and $\tilde{\Q}$ both have the same $\ell^2$-norm.

Next we consider the general case. 
Using the trivial estimate and the previous special case we obtain
\begin{alignat*}1
\max_{1\leq i\leq d}\vert \Q(\xi_i)\vert&\geq \sup_{\phi}\left(\max_{1\leq j\leq d}|\Q(\phi(\zeta_d^j))| -\max_{1\leq j\leq d}\min_{1\leq i\leq d}|\Q(\xi_i)-\Q(\phi(\zeta_d^j))|\right)\\
&\geq \sqrt{\sum_{0\leq i\leq n}\vert b_{i}\vert^2}-\inf_{\phi}\max_{1\leq j\leq d}\min_{1\leq i\leq d}|\Q(\xi_i)-\Q(\phi(\zeta_d^j))|.
\end{alignat*}
Hence, it suffices to show
\begin{alignat}1\label{ineq:error}
\inf_{\phi}\max_{1\leq j\leq d}\min_{1\leq i\leq d}|\Q(\xi_i)-\Q(\phi(\zeta_d^j))|
\leq D(\vxi)d^{3/2}\max_{1\leq i\leq d}\{1,|\xi_i|\}^{d-2}\sqrt{\sum_{0\leq i\leq n}\vert b_{i}\vert^2}.
\end{alignat}

Let $\gamma$ denote the (complex) line segment
connecting a point $\zeta\in \IT$ with an arbitrary point $\xi$.
We use the parametrisation $\gamma(\tau)=(1-\tau)\zeta+\tau \xi$ with $\tau\in [0,1]$.
By using the complex line integral of $\Q$ along $\gamma$, we write
$\Q(\xi)-\Q(\zeta)=\int_{\gamma} \Q'(\tau) ~\mathrm{d}\tau$. Hence, 
\begin{alignat}1\label{ineq:f}
|\Q(\xi)-\Q(\zeta)| \leq \vert \xi-\zeta \vert \max_{ \tau \in [0,1]}\vert \Q'(\gamma(\tau))\vert .
\end{alignat}
Further, as $\vert \zeta \vert \leq \max\{ 1, \vert \xi \vert \} $, we have
$ \vert \gamma(\tau) \vert \leq  \max\{1,|\xi|\}$
for any $\tau\in[0,1]$. By the Cauchy--Schwarz inequality,
\begin{alignat}1\label{ineq:f'}
\vert \Q'(\gamma(\tau))\vert 
\leq n \sqrt{\sum_{0\leq i < n}\vert \gamma(\tau)\vert^{2 i}} 
          \sqrt{\sum_{0\leq i < n}\vert b_{i}\vert^2}
\leq d^{3/2}\max\{1,|\xi|\}^{d-2}\sqrt{\sum_{0\leq i\leq n}\vert b_{i}\vert^2}.
\end{alignat}
Combining (\ref{ineq:f}) and (\ref{ineq:f'}), with $\xi=\xi_i$ and $\zeta=\phi(\zeta_d^j)$, 
while using $\max\{1,|\xi_i|\}\leq \max_{1\leq i \leq d}\{1,|\xi_i|\}$
the inequality \eqref{ineq:error} drops out. This proves the lemma.
\end{proof}

Next we derive a corollary. Recall that $\|\cdot\|$ denotes the maximum norm on $\IC^d$.

\begin{corollary}
\label{cor: ell two lower bound}
Let $\vxi\in \IC^d\setminus\{\mathbf{0}\}$ and suppose that
$\Q(x)=b_{0}+b_{1}x+\cdots+b_{n}x^{n}\in\mathbb{C}[x]$
has degree strictly less than $d$. We have 
\[
\max_{1\leq i\leq d}\vert \Q(\xi_i)\vert\geq \left(1-d^{3/2}D\left(\frac{\vxi}{\|\vxi\|}\right)\right)\sqrt{\sum_{0\leq i\leq n}\vert b_{i}\cdot\|\vxi\|^i\vert^2},
\]
If $n=0$ then we can omit the first factor.
\end{corollary}
\begin{proof}
Apply Lemma \ref{lem: ell two lower bound} with $\vxi$ and $\Q(x)=b_{0}+b_{1}x+\cdots+b_{n}x^{n}$ replaced by $\frac{\vxi}{\|\vxi\|}$
and $\tilde{\Q}(x)=\Q(\|\vxi\| x)$.
\end{proof}

 Decomposing $\Q(x)=\sum_{j\in J}b_j x^{j}+\sum_{k\notin J}b_k x^{k}$ and applying 
 Lemma \ref{lem: ell two lower bound} or Corollary \ref{cor: ell two lower bound} to  $\sum_{j\in J}b_j x^{j}$ sometimes
 allows to produce non-trivial results, even when $\deg \Q\geq d$. 
 Let us record here only  the special case when all $\xi_i$ lie on the unit circle.
\begin{corollary}
\label{cor: ell two lower bound two}
Let $\vec \xi \in \IT^d$, and 
suppose $\Q(x)=b_{0}+b_{1}x+\cdots+b_{n}x^{n}\in\mathbb{C}[x]$. 
Then for every $I\subset \{1,2,\ldots,d\}$ and each non-empty $J\subset \{0,1,\ldots,n\}$
with $\max_{j\in J} j -\min_{j\in J} j<\#I$ we have
\[
\max_{1\leq i\leq d}\vert \Q(\xi_i)\vert\geq\big(1-(\#I)^{3/2}D_I\big)\sqrt{\sum_{j\in J}\vert b_{j}\vert^2}-\sum_{k\notin J}\vert b_{k}\vert.
\]
where $D_I=D(({\xi_i})_{i\in I})$. 
\end{corollary}

\section{Application to Galois orbits and lower bounds for the house}\label{sec:Galoisorbits}
In this section we apply the results of Section \ref{sec:geometriclowerbound} to a diophantine setting. Throughout this section
let $K$ be subfield of $\Qbar$, $\xi\in \Qbar\backslash\{0\}$, and $d=[K(\xi):K]$. Let us write 
$$M_{\xi,K}=(x-\xi_1)\cdots(x-\xi_d)\in K[x]$$ 
for the monic minimal polynomial 
of $\xi$ over $K$ with conjugates $\xi_1,\ldots,\xi_d$ over $K$. Recall that
$$\mathrm{Hom}(K)=\{\sigma:K\to \IC; \text{ field homomorphism}\}.$$
For each $\sigma\in\mathrm{Hom}(K)$  let 
$$\tau_{1,\sigma},\ldots,\tau_{d,\sigma}\in \mathrm{Hom}(K(\xi))$$ 
be the $d$ extensions of $\sigma$ from $K$
to $K(\xi)$, and set
\begin{alignat*}1
\vtau_\sigma(\xi)=(\tau_{1,\sigma}(\xi),\ldots,\tau_{d,\sigma}(\xi)).
\end{alignat*}
So the components of $\vtau_\sigma(\xi)$ are precisely the $d$ distinct roots of the irreducible polynomial $\sigma(M_{\xi,K})\in \sigma(K)[x]$.

\begin{lemma}
\label{lem:housebound}
Let $\Oseen$ be a subring of $\Oseen_K$, and
let $\xi$ be a non-zero algebraic integer. Suppose that $P(x)=a_{0}+ a_{1}x+\cdots + a_{n}x^{n} \in \Oseen[x]$ is of degree
$1\leq n<d$, and let  $ \alpha=P(\xi)$.
Then  we have
$$\housealp \geq \max_{\sigma\in \mathrm{Hom}(K)}\,
\left(1-d^{3/2}D\left(\frac{\vtau_\sigma(\xi)}{\|\vtau_\sigma(\xi)\|}\right)\right)\,
\left|\sigma(a_{n})\right|\|\vtau_\sigma(\xi)\|^{n}$$
\end{lemma}
\begin{proof}
We have 
\begin{align*}
\housealp & \geq\max_{\sigma\in \mathrm{Hom}(K)}\,\max_{\underset{\tau\mid_{K}=\sigma}{\tau\in \mathrm{Hom}(K(\xi))}}\left|\tau(P(\xi))\right|
=\max_{\sigma\in \mathrm{Hom}(K)}\,\max_{\underset{\tau\mid_{K}=\sigma}{\tau\in \mathrm{Hom}(K(\xi))}}\left|\sigma(P)(\tau(\xi))\right|.
\end{align*}
Due to Corollary \ref{cor: ell two lower bound}, the right hand side is at least
\begin{equation}\label{final}
\max_{\sigma\in \mathrm{Hom}(K)}\,
\left(1-d^{3/2}D\left(\frac{\vtau_\sigma(\xi)}{\|\vtau_\sigma(\xi)\|}\right)\right)\,
\left(\sum_{0\leq k\leq n}
\left|\sigma(a_{k})\|\vtau_\sigma(\xi)\|^{k}\right|^{2}\right)^{1/2}.
\end{equation}
The second factor is, trivially, at least $\left|\sigma(a_{n})\right|\|\vtau_\sigma(\xi)\|^{n}$,
which proves the claim. 
\end{proof}

We now introduce a  new invariant for the tuple $(K,\xi)$.
\begin{definition}\label{def:etainvariant}
Let
\begin{alignat}1
\eta(K,\xi)=\min_{\sigma\in\mathrm{Hom}(K)}\min\{\|\vtau_\sigma(\xi)\|,\|\vtau_\sigma(\xi)\|^{d-1}\}\left(1-d^{3/2}D\left(\frac{\vtau_\sigma(\xi)}{\|\vtau_\sigma(\xi)\|}\right)\right).
\end{alignat}
\end{definition}
\begin{remark}\label{rem:etainvariant}
The quantity $\eta(K,\xi)$ can also be expressed 
in terms of the minimal polynomial $M_{\xi,K}$. 
For  $M=(x-\alpha_1)\cdots(x-\alpha_d)\in \IC[x]\backslash \IC$, not a monomial, define
$$\eta_0(M)=\min\{\|\valpha\|,\|\valpha\|^{d-1}\}\left(1-d^{3/2}
D\left(\frac{\valpha}{\|\valpha\|}\right)\right),$$
where $\valpha=(\alpha_1,\ldots,\alpha_d)\in \IC^d\setminus\{\mathbf{0}\}$. We have
$$\eta(K,\xi)=\min_{\sigma\in \Hom(K)}\eta_0(\sigma(M_{\xi,K})),$$
where $\sigma(M_{\xi,K})$ is understood to be applied coefficient-wise.
This shows that $\eta(K,\xi)=\eta(L,\xi)$ whenever $M_{\xi,K}=M_{\xi,L}$. 
In particular, if $[K(\xi):K]=[\IQ(\xi):\IQ]$ then 
\begin{alignat}1
\eta(K,\xi)=\eta(\IQ,\xi).
\end{alignat}
\end{remark}

Our next result is  key to determine the Northcott number of various examples, including those leading to Theorem \ref{thm: spectrum full for house function}.
Roughly speaking,  $\eta(\cdot,\cdot)$ will provide a lower bound for the house 
of ``new elements'' in extensions as required to apply the general Lemma \ref{LemmaN}. 

\begin{proposition}[Key Proposition]
\label{prop:housebound}
Let $\Oseen$ be a subring of $\Oseen_K$, 
let $\xi$ be an algebraic integer but not in $\Oseen_K$, and suppose that
$M_{\xi,K}\in \Oseen[x]$. Then we have
$$\delta_{\house{\cdot}}(\Oseen[\xi]\backslash \Oseen)\geq \eta(K,\xi).$$
\end{proposition}
\begin{proof}
Suppose $\alpha \in \Oseen[\xi]\backslash \Oseen$.
Since $M_{\xi,K}\in \Oseen[x]$ we
can choose a polynomial $P\in \Oseen[x]$ of degree
$1\leq n<d$ such that 
$ \alpha=P(\xi)=a_{0}+ a_{1}\xi+\cdots + a_{n}\xi^{n}$. From Lemma \ref{lem:housebound} we conclude that 
$$\housealp \geq \max_{\sigma\in \mathrm{Hom}(K)}\,
\left(1-d^{3/2}D\left(\frac{\vtau_\sigma(\xi)}{\|\vtau_\sigma(\xi)\|}\right)\right)\,
\left|\sigma(a_{n})\right|\|\vtau_\sigma(\xi)\|^{n}.$$
As $a_{n}\neq 0$ is an algebraic integer there exists $\sigma\in \mathrm{Hom}(K)$
such that $\left|\sigma(a_{n})\right|\geq1$. Noticing that  $\|\vtau_\sigma(\xi)\|^n\geq \min\{\|\vtau_\sigma(\xi)\|,\|\vtau_\sigma(\xi)\|^{d-1}\}$ completes the proof. 
\end{proof}
For any eligible choices of $K$ the hypothesis ensures that 
the minimal polynomials of $\xi$ over $K$, and over the field of fractions of $\Oseen$ are identical.
Hence $\eta(K,\xi)$ is independent of the particular choice of $K$.

\section{Proof of Theorem \ref{thm:Nnumbergeneral}}
With Proposition \ref{prop:housebound} at hand, we can now easily prove Theorem \ref{thm:Nnumbergeneral}.
Recall  that $(\xi_i)_i$ is a sequence of  algebraic integers, $\Oseen_0$ is a ring (containing $1$) of algebraic integers,
$$\Oseen_i=\Oseen_0[\xi_1,\ldots,\xi_i],
\,\, \mathrm{and}
\,\, \Oseen=\bigcup_{i\geq 1} \Oseen_i=\Oseen_0[\xi_1,\xi_2,\xi_3,\ldots].
$$
The field $K_i$ is the field of fractions of $\Oseen_i$, and  
by hypothesis $M_{\xi_i,K_{i-1}}\in \Oseen_{i-1}[x]$. Moreover,  $d_i=[K_{i-1}(\xi_i):K_{i-1}]>1$, and thus $\xi_i\notin \Oseen_{K_{i-1}}$.
From Proposition \ref{prop:housebound} we conclude that
$$\delta_{\house{\cdot}}(\Oseen_{i-1}[\xi_i]\backslash \Oseen_{i-1})\geq \eta(K_{i-1},\xi_i).$$
Applying Lemma \ref{LemmaN}  with $A_i=\Oseen_i$ proves
Theorem \ref{thm:Nnumbergeneral}.

\section{Applications of the equidistribution method}\label{sec:applications}

In this section we discuss a few special cases of Theorem \ref{thm:Nnumbergeneral}. 
Recall that by Northcott's Theorem (cf. Lemma \ref{lem:generalNorthcott}) one has
$\N_{\house{\cdot}}(\Oseen)=\infty$
for any ring of algebraic integers $\Oseen$,
whose field of fractions has finite degree over $\IQ$.
We will use this fact without further notice.

\begin{corollary}\label{example:standardfields}
Let  $(p_{i})_{i\in \IN}$, and $(d_{i})_{i\in \IN}$ be two sequences of primes, and suppose the primes $d_i$ are pairwise distinct.
Let $\xi_i=p_i^{1/d_i}$ be any $d_i$-th root, $\Oseen_i=\IZ[\xi_1,\ldots, \xi_i]$, $K_i$ be the field of fractions of $\Oseen_i$, and
let $\Oseen=\cup_{i\geq 1}\Oseen_i$. 
Then  $\eta(K_{i-1},\xi_i)=\eta(\IQ,\xi_i)=\house{\xi_i}$, and
$\N_{\house{\cdot}}(\Oseen)= \liminf_{i\to \infty} \house{\xi_i}$.
\end{corollary}
\begin{proof}
Let  $\liminf _{i\to \infty}\house{\xi_i}=t\in [1,\infty]$. 
Since the $\xi_i$ are pairwise distinct,  we get $\N_{\house{\cdot}}(\Oseen)\leq t$. 
We derive, by Eisenstein's criterion and the tower-law,
that $d_i=[K_{i-1}(\xi_{i}):K_{i-1}]=[\IQ(\xi_{i}):\IQ]$.
Hence, $M_{\xi_i,K_{i-1}}=M_{\xi_i,\IQ}\in \Oseen_{i-1}[x]$, and $\eta(K_{i-1},\xi_i)=\eta(\IQ,\xi_i)$.
As the  $d_i$ conjugates $\house{\xi_i}\;\zeta_{d_i}^j$ of $\xi_i$ are perfectly equidistributed on the 
circle of radius $\house{\xi_i}$ about the origin, we have $\eta(\IQ,\xi_i)=\house{\xi_i}$.
Thus, Theorem \ref{thm:Nnumbergeneral}
implies $\N_{\house{\cdot}}(\Oseen)\geq t$.
\end{proof}

Next we consider a generalisation of Corollary \ref{example:standardfields}. To reduce clutter we define for each non-zero algebraic number $\xi$
the unordered tuple of normalised conjugates  
$$\vcxi=\left(\frac{\xi_1}{\house{\xi}},\ldots,\frac{\xi_d}{\house{\xi}}\right),$$
where  $d=[\IQ(\xi):\IQ]$ and $\xi_1,\dots,\xi_d$ are the conjugates of $\xi$.
We remind the reader that $D(\vcxi)$ is well-defined as the finite discrepancy $D(\cdot)$ is indifferent to the order of the components. 
Before stating the result let us informally discuss what we would get by applying Theorem \ref{thm:Nnumbergeneral} in the most naive way.

Let $(\gamma_i)_i$ be a sequence of algebraic integers of degree $m_i$, let $\xi_i=\gamma_i^{1/n_i}$
be a $n_i$-th root of $\gamma_i$, and suppose that $[K_{i-1}(\xi_i):K_{i-1}]=m_in_i>1$. Applying Theorem \ref{thm:Nnumbergeneral} yields
$$\N_{\house{\cdot}}(\Oseen)\geq \liminf_{i\rightarrow \infty} \eta(K_{i-1},\xi_i)= \liminf_{i\rightarrow \infty} \eta(\IQ,\xi_i)=\liminf_{i\rightarrow \infty} \house{\xi_i}\left(1-(m_in_i)^{3/2}D\left(\vcxii\right)\right).$$
If $\lim_{i\to \infty} (m_in_i)^{3/2}D\left(\vcxii \right)=0$  then we  conclude that $\N_{\house{\cdot}}(\Oseen)=\liminf_{i \rightarrow \infty} \house{\xi_i}$.
In general the best shot at the former equation we have is Lemma \ref{lem:discrepest1} (further below in this section), which ensures the required condition  provided 
$$\lim_{i\to \infty} m_i^{3/2}n_i^{1/2}D\left(\vcgammai\right)=0.$$
Hence, we need the  normalised conjugates $\vcgammai$ of $\gamma_i$ to converge very rapidly 
to a perfect equidistribution as $i$ gets large.  

However, Theorem \ref{thm:Nnumbergeneral} allows us to consider the conjugates of $\gamma_i$ over $K_{i-1}$
and of $\xi_i=\gamma_{i}^{1/n_i}$ over $K_{i-1}(\gamma_i)$ separately.  Now the conjugates of  $\xi_i$ over $K_{i-1}(\gamma_i)$ 
are perfectly equidistributed and those of $\gamma_i$ over $K_{i-1}$ are much easier to control than those of $\xi_i$ over $\IQ$.
In this way we can relax the above condition to
$$m_i^{3/2}D\left(\vcgammai\right)\leq 1-\house{\gamma_i}^{1/n_i-1}$$
for all sufficiently large $i$.
Here is the precise result.

\begin{corollary}\label{cor:Nnumber}
Let $(\gamma_i)_i$ be a sequence of algebraic integers of respective degree $m_i$ with conjugates $\gamma_1^{(i)},\ldots,\gamma_{m_i}^{(i)}$,
and with  smallest (complex) absolute value  $s_i=\min_{1\leq j\leq m_{i}}|\gamma_{j}^{(i)}|$. Let $(n_i)_i$ be a sequence of integers $>1$, and choose a $n_i$-th root 
$\gamma_i^{1/n_i}$ for each $i\in \IN$.
Suppose that the following four properties hold 
\begin{itemize}
\item $[\IQ(\gamma_1^{1/n_1},\ldots,\gamma_i^{1/n_i}):\IQ]=m_1n_1\cdots m_i n_i$ for each $i\in \IN $.
\item $D(\vcgammai)\leq m_i^{-3/2}(1-\house{\gamma_i}^{1/n_i-1})$ for all but finitely many $i$.
\item $s_i\geq 1$ for all but finitely many $i$.
\item $\displaystyle{\lim_{i\to \infty} \left(\frac{s_i}{\house{\gamma_i}}\right)^{1/n_i}=1}$.
\end{itemize}
With $\Oseen=\IZ[\gamma_i^{1/n_i};  i\in \IN ]$, we have $$\N_{\house{\cdot}}(\Oseen)=\liminf_{i \rightarrow \infty} \house{\gamma_i}^{1/n_i}.$$
\end{corollary}
\begin{proof}
Let $i_0\geq 1$ be an index
such that the second and third condition is satisfied
 for all $i\geq i_0$, and set $\Oseen_0=\IZ[\gamma_i^{1/n_i}; i\leq i_0]$.
For $i\in \IN$ set $\xi_{2i-1}=\gamma_{i+i_0}$ and $\xi_{2i}=\gamma_{i+i_0}^{1/n_{i+i_0}}$.
Since  $n_{i+i_0}>1$ the $\xi_{2i}$ are pairwise distinct. Thus 
$\N_{\house{\cdot}}(\Oseen)\leq \liminf_{i \rightarrow \infty} \house{\xi_{2i}}$.
 
To prove the reversed inequality we let $\Oseen_i=\Oseen_0[\xi_{1},\ldots,\xi_{i}]$, and we write $K_i$ for the field of fractions of $\Oseen_i$. 
The first condition implies that the minimal polynomial 
$M_{\xi_{2i-1},K_{2i-2}}$ has coefficients in $\IZ$, and that 
$$ 
M_{\xi_{2i},K_{2i-1}}=x^{n_{i+i_0}}-\gamma_{i+i_0}=x^{n_{i+i_0}}-\xi_{2i-1}.
$$ 
Hence, 
$M_{\xi_{i},K_{i-1}}\in \Oseen_{i-1}[x]$ for all $i\in \IN $, and $\xi_{2i}\notin \Oseen_{K_{2i-1}}$ as $n_{i+i_0}>1$.
Moreover, if $\xi_{2i-1}\in \Oseen_{K_{2i-2}}$ then $m_{i+i_0}=1$ and thus $\xi_{2i-1}\in \IZ\subset \Oseen_{0}$.
Hence,  we can assume without loss of generality  that $\xi_i$ is not in $\Oseen_{K_{i-1}}$.
Thus we can apply Theorem \ref{thm:Nnumbergeneral},
and it suffices to check that 
$\liminf_{i\rightarrow \infty} \eta(K_{i-1},\xi_i) \geq \liminf_{i \rightarrow \infty} \house{\xi_{2i}}$.
Note that the first condition implies 
$[K_{2i-2}(\xi_{2i-1}):K_{2i-2}]=[\IQ(\xi_{2i-1}):\IQ]$. Hence, 
\begin{alignat*}1
\eta(K_{2i-2},\xi_{2i-1})=\eta(\IQ,\xi_{2i-1})=\house{\gamma_{i+i_0}}\big(1-m_{i+i_0}^{3/2}D\left(\vcgammainull\right)\big)\geq \house{\gamma_{i+i_0}^{1/n_{i+i_0}}}=\house{\xi_{2i}},
\end{alignat*}
using the second hypothesis.

Finally, note that by the first hypothesis each $\sigma\in \Hom(K_{2i-1})$ maps $\gamma_{i+i_0}$ to a conjugate $\gamma_k^{(i+i_0)}$, and thus the $n_{i+i_0}$ extensions $\tau_{j,\sigma}$ of $\sigma$
map $\xi_{2i}$ to the $n_{i+i_0}$ perfectly equidistributed complex numbers  ${\gamma_k^{(i+i_0)}}^{1/n_{i+i_0}}\zeta_{n_{i+i_0}}^l$ with $1\leq l\leq n_{i+i_0}$. 
Therefore, using the third condition, we have for all $i\in \IN$ that
$$\eta(K_{2i-1},\xi_{2i})=\min\left\{s_{i+i_0},s_{i+i_0}^{n_{i+i_0}-1}\right\}^{1/n_{i+i_0}}=s_{i+i_0}^{1/n_{i+i_0}}=\left(\frac{s_{i+i_0}}{\house{\gamma_{i+i_0}}}\right)^{1/n_{i+i_0}} {\house{\gamma_{i+i_0}}}^{1/n_{i+i_0}}=\left(\frac{s_{i+i_0}}{\house{\gamma_{i+i_0}}}\right)^{1/n_{i+i_0}} {\house{\xi_{2i}}},$$
and thus by the fourth hypothesis $\liminf_{i \rightarrow \infty} \eta(K_{2i-1},\xi_{2i})\geq \liminf_{i \rightarrow \infty} \house{\xi_{2i}}$.

\end{proof}

Next, note that Theorem \ref{thm:Nnumbergeneral} requires that all  $\Oseen_i$ have the Northcott property (with respect to the $\house{\cdot}$)
but nonetheless $\Oseen_0$ may be an infinite ring extension of $\IZ$. J. Robinson has shown that the ring of integers of the composite field of all real quadratic
number fields has the Northcott property. Various other examples have been established in \cite{BoZa, CheccoliFehm, CheccoliWidmer,  WidmerPropN}.
Theorem \ref{thm:Nnumbergeneral} allows to extend the Northott property of these rings $\Oseen_0$ to certain bigger rings, 
or alternatively, to create  ring extensions of these $\Oseen_0$ with prescribed finite Northcott number.

\begin{corollary}\label{example:largerings1}
Let $\Oseen_0$ be a ring of algebraic integers, and let $(\xi_i)_i$ be pairwise distinct algebraic integers. Let
$\Oseen_i=\Oseen_0[{\xi_1},\ldots,{\xi_i}]$, let $K_i$ be the field of fractions of $\Oseen_i$, set
$\Oseen=\Oseen_0[\xi_1,\xi_2,\xi_3,\ldots]$, and suppose that $[K_{i-1}(\xi_i):K_{i-1}]=[\IQ(\xi_i):\IQ]=d_i>1$. 
If $\N_{\house{\cdot}}(\Oseen_{K_0})=\infty$, then 

$$\liminf_{i\to \infty} \house{\xi_i}\geq \N_{\house{\cdot}}(\Oseen)
\geq \liminf_{i\to \infty} \house{\xi_i}\left(1-d_i^{3/2}
D\left(\vcxii\right)\right).
$$
In particular,  if $n_i$ are rational integers, and the elements $\xi_i=n_i^{1/d_i}$ 
have degree $d_i>1$ over $K_{i-1}$, then $ \N_{\house{\cdot}}(\Oseen)
= \liminf_{i\to \infty} |n_i^{1/d_i}|$.
\end{corollary}
\begin{proof}
First note that  by Lemma \ref{lem:generalNorthcott} (cf. Remark \ref{rem:sec2})  we conclude $\N_{\house{\cdot}}(\Oseen_{K_i})=\infty$, and thus $\N_{\house{\cdot}}(\Oseen_i)=\infty$.
Now note that
$$\eta(K_{i-1},\xi_i)=\eta(\IQ,\xi_i)=\house{\xi_i}\left(1-d_i^{3/2}D\left(\vcxii\right)\right),$$
and $M_{\xi_i,K_{i-1}}=M_{\xi_i,\IQ}$.
Hence, the claim follows from Theorem \ref{thm:Nnumbergeneral}.
\end{proof}

We end this section with some simple estimates for the finite discrepancy, one of which (Lemma \ref{lem:discrepest1}) has already been used in the paragraph
after the proof of Corollary \ref{example:standardfields} to motivate Corollary \ref{cor:Nnumber}. We believe they might also be useful 
for further applications of Theorem \ref{thm:Nnumbergeneral}. 
We first introduce a slight refinement of the finite discrepancy. For $\vxi\in \IC^d$ and $\u\in \IT$
we set 
\begin{alignat*}1
D_\u(\vxi)=\max_{1\leq j\leq d}\min_{1\leq i\leq d}|\xi_i-\u\zeta_d^j|,
\end{alignat*}
so that
\begin{alignat}1\label{eq:discrepinf}
D(\vxi)=\displaystyle{\inf_{\u\in \IT}D_\u(\vxi)}.
\end{alignat}
For fixed $\vxi\in \IC^d$ the function $D_\u(\vxi)$ is continuous in $\u$. Since $\IT$ is compact if follows 
that there exists $\u\in \IT$ such that
\begin{alignat}1\label{eq:discrepspec}
D(\vxi)=D_\u(\vxi). 
\end{alignat}

\begin{lemma}\label{lem:discrepest1}
Let $n,m\in\IN$. Let $\vec{\alpha}=(\alpha_1,\ldots,\alpha_m)\in \IC^m$,  and  suppose that $D\left({\vec{\alpha}}\right)\leq 1/2$. Then we have
$$D\left(\left(\left({\alpha_k}\right)^{1/n}\zeta_n^\ell\right)_{1\leq k\leq m  \atop{1\leq \ell \leq n}}\right)\leq \frac{2}{n}D\left({\vec{\alpha}}\right).$$
\end{lemma}
Note that the left hand-side is independent of the particular choice of the $n$-th root.

\begin{proof}
By (\ref{eq:discrepspec}) we can choose $\u\in \IT$ such that $D_\u\left({\vec\alpha}\right)=D\left({\vec\alpha}\right)$.
By (\ref{eq:discrepinf}) it suffices to prove that
\begin{equation}\label{discrepancy-eta-ineq1}
D_{\u^{1/n}}
\left(\left(\left({\alpha_k}
\right)^{1/n}\zeta_n^\ell\right)_{1\leq k\leq m  \atop{1\leq \ell \leq n}}\right)\leq 
\frac{2}{n}D_\u\left({\vec\alpha}\right).
\end{equation}

After  reordering the coordinates of $\valpha$, there exists 
for each $1\leq k \leq m$
a real number $\delta$ satisfying $0\leq \delta \leq D_\u(\valpha)\leq \frac{1}{2}$, and a real number $\varphi$ such that 
\begin{equation*}
{\alpha_k} = \zeta_m^k \u + \delta e^{i\varphi}.
\end{equation*}
Thus we obtain

\begin{equation*}
{\alpha_k}^{1/n} = (\zeta_m^k \u + \delta e^{i\varphi})^{1/n}=\zeta_{mn}^{k+sm}\; \u^{1/n}\;(1+\delta e^{i\varphi'})^{1/n},
\end{equation*}

\noindent for some integer $1\leq s\leq n$ and some real number $\varphi'$. The generalised binomial theorem implies for any complex number $z$,
with $\left|z\right|<1$, the Taylor expansion
\[
(1+z)^{1/n}=1+\sum_{r\geq1}\begin{pmatrix}1/n\\
r
\end{pmatrix}z^{r}\quad\mathrm{where}\quad\begin{pmatrix}1/n\\
r
\end{pmatrix}=\frac{\frac{1}{n}\left(\frac{1}{n}-1\right)\ldots\left(\frac{1}{n}-(r-1)\right)}{r!}\,\,\mathrm{for}\,r\geq1.
\]
Further,
\[
\left|\begin{pmatrix}1/n\\
r
\end{pmatrix}\right|=\frac{\frac{1}{n}}{1}\frac{\left(1-\frac{1}{n}\right)}{2}\ldots\frac{(r-1)-\frac{1}{n}}{r}\leq\frac{1}{n}\quad(r\geq1).
\]
Consequently, 
\[
\left|(1+\delta e^{i\varphi'})^{1/n}-1\right|\leq\frac{1}{n}\sum_{r\geq1}\delta^{r}=\frac{1}{n}\frac{\delta}{(1-\delta)}\leq \frac{2\delta}{n}.
\]

Finally, note that for any fixed $1\leq s\leq n$, the expression $\zeta_{mn}^{k+sm} \u^{1/n} \zeta_n^\ell = \zeta_{mn}^{k+(\ell+ s) m} \u^{1/n}$, for $1\leq k \leq m$ and $1\leq \ell \leq n$ ranges over all elements $\zeta_{nm}^r \u^{1/n}$ for $1\leq r\leq mn$. This leads to (\ref{discrepancy-eta-ineq1}), hence gives the result.

\end{proof}

\begin{lemma}\label{lem:discrepest2}
Let $n,m\in\IN$ be coprime, let $\vec{\alpha}=(\alpha_{1},\ldots,\alpha_{m})\in \IC^m$, $\vec{\beta}=(\beta_{1},\ldots,\beta_{n})\in \IC^n$.
Then we have
$$D\left((\alpha_k\beta_{l})_{1\leq k\leq m \atop {1\leq l\leq n}}\right)\leq (1+D(\vec{\alpha}))(1+D(\vec{\beta}))-1.$$
\end{lemma}

\begin{proof}
By (\ref{eq:discrepspec}) we can choose $\u, \u'\in \IT$ such that $D_\u\left({\vec\alpha}\right)=D\left({\vec\alpha}\right)$ and $D_{\u'}\left({\vec\beta}\right)=D\left({\vec\beta}\right)$. By (\ref{eq:discrepinf}) it suffices to prove that
\begin{equation}\label{discrepancy-eta-ineq2}
D_{\u\u'}((\alpha_k\beta_{\ell})_{1\leq k\leq m \atop {1\leq \ell\leq n}}) \leq (1+D_{\u}(\vec \alpha))(1+D_{\u'}(\vec \beta))-1.
\end{equation}
After reordering the coordinates of $\valpha$ and $\vbeta$ we can write
$\alpha_{k}=\zeta_{m}^{k}\u+\delta_{k}e^{i\varphi_k}$ and $\beta_{\ell}=\zeta_{n}^{\ell}\u'+\delta_{\ell}'e^{i\varphi_\ell'}$
for certain non-negative real  $\delta_{k},\delta_{\ell}'$ and $\varphi_{k}, \varphi_{\ell}'$ satisfying
$0\leq \delta_{k}\leq D_{\u}(\vec \alpha),\, 0\leq \delta_{\ell}'\leq D_{\u'}(\vec \beta)$.

Further
\[
(\zeta_{m}^{k}\u+\delta_{k}e^{i\varphi_k})(\zeta_{n}^{\ell}\u'+\delta_{\ell}'e^{i\varphi_\ell'})=\zeta_{mn}^{kn+\ell m}\u\u'+\epsilon_{k, \ell}
\]
where $|\epsilon_{k, \ell}|\leq \delta_{k}
+\delta_{\ell}'+\delta_{k}\delta_{\ell}'
\leq (1+D_{\u}(\valpha))(1+D_{\u'}(\vbeta))-1$. 
By coprimality of $m$ and $n$, the exponent $kn+\ell m$ 
represents all residue classes modulo $mn$ as $1\leq k \leq m$
and $1\leq\ell\leq n$. This proves  (\ref{discrepancy-eta-ineq2}), and hence
the lemma.
\end{proof}

\section{From arbitrary rings of algebraic integers to integrally closed subrings}

In this section we record two classical results that provide a strategy to show that 
a ring of algebraic numbers is integrally closed, i.e., is the ring of integers of a field.

In the first step we need to identify when a given algebraic integer $\theta$ of $K$ generates
the ring of integers $\Oseen_K$ over $\IZ$. There are at least three  such criteria in the literature - Dedekind's, Uchida's
and L\"uneburg's criterion (cf. \cite{VV21}).
For our purpose Dedekind's classical criterion is well suited (cf. \cite[Theorem 1.1]{MunishKhanduja}).

\begin{lemma}[Dedekind's criterion]
\label{lem: dedekind's criterion}Let $\theta$ be an algebraic integer,
and let $M_{\theta,\IQ}$ be its minimal polynomial over $\IQ$. Let $K=\mathbb{Q}(\theta)$, and 
let $q$ be a rational prime. Let $\overline{M_{\theta,\IQ}}$ 
be the reduction of $M_{\theta,\IQ}$ modulo q and 
$\overline{M_{\theta,\IQ}}=\varphi_{1}^{e_{1}}\ldots\varphi_{k}^{e_{k}}$
be the decomposition of $\overline{M_{\theta,\IQ}}$ into irreducible factors over
the ring of polynomials $\mathbb{F}_{q}[x]$ over the field with $q$
elements. Let $\mu_{i}, g\in\mathbb{Z}[x]$ be such that
\[
M_{\theta,\IQ}=\mu_{1}^{e_{1}}\ldots\mu_{k}^{e_{k}}+qg
\]
and $\overline{\mu}_{i} = \varphi_{i}$ for all $i\leq k$.
The following are equivalent:
\begin{enumerate}
\item The prime $q$ does not divide $[\Oseen_{K}:\mathbb{Z}[\theta]]$.
\item For all $i\leq k$, either $e_{i}=1$, or $\varphi_i\nmid\overline{g}$
in $\mathbb{F}_{q}[x]$.
\end{enumerate}
\end{lemma}

Next we require  a criterion to factor the ring of integers of a compositum 
into a product of rings of integers.
If $\Oseen$ is an order  of a number field $K$ then we write 
$\Delta_{\Oseen}$ for the discriminant
of that order, and just $\Delta_K$ if $\Oseen=\Oseen_K$ is the maximal order.  
For subfields $K$ and $F$ of $\Qbar$
we write $KF$ for their composite field.
The next statement can be found in \cite[Proposition III.2.13]{FroTay}.
\begin{lemma}\label{lem:compositeringofintegers}
If $F,K$ are number fields and  linearly disjoint  over $\mathbb{Q}$,
with $(\Delta_{F},\Delta_{K})=1$, then $\Oseen_{KF}=\Oseen_{K}\cdot \Oseen_{F}$
where the right hand side denotes
the smallest subring of $\Oseen_{\Qbar}$
containing $\Oseen_{K}$ and $\Oseen_{F}$.
\end{lemma}

We will apply these general criteria to rings of the form $\IZ[p_1^{1/d_1},p_2^{1/d_2},p_3^{1/d_3},\ldots]$,
where $d_i$ and $p_i$ are suitably chosen primes. To this end we first need to compute the modulus of the discriminant
$\Delta_{\mathbb{Z}[p_i^{1/d_i}]}$.

\begin{lemma}\label{lem:disc}
Let $n\in \IN$, suppose $x^d-n\in \IZ[x]$ is irreducible,
 and let $\theta\in \Qbar$ be one of its roots.
We have $\left|\Delta_{\mathbb{Z}[\theta]}\right|=d^{d}n^{d-1}$. 
\end{lemma}

\begin{proof}
Note that the minimal polynomial $M_{\theta,\IQ}$ of $\theta$ 
is given by $\prod_{i\leq d}\left(x-\zeta_{d}^{i}\theta\right)$.
Thus 
\begin{alignat}1\label{eq:lem:disc}
\left|\Delta_{\mathbb{Z}[\theta]}\right|=
\left|\prod_{\substack{1 \leq i,j \leq d \\ i \neq j}} 
(\zeta_{d}^{j}\theta-\zeta_{d}^{i}\theta)\right|
=\vert \theta^{d(d-1)}\vert \left|\prod_{\substack{1 \leq i,j \leq d \\ i \neq j}} 
(1-\zeta_{d}^{i-j})\right|.
\end{alignat}
Using the basic identity 
\[
\prod_{i<d}\left(x-\zeta_{d}^{i}\right)=\frac{x^{d}-1}{x-1}=\sum_{k=0}^{d-1}x^{k},
\]
and that for any fixed $i\leq d$ the relation 
$\{ \zeta_{d}^{i-j};\,j\leq d,\,j\neq i\} =\{ \zeta_{d}^{j};\,j\leq d-1\} $
holds, we see that 
\[
\prod_{\substack{1 \leq i,j \leq d \\ i \neq j}}
(1-\zeta_{d}^{i-j})=
\left(\prod_{i<d}\left(1-\zeta_{d}^{i}\right)\right)^{d}=d^{d}.
\]
The latter in conjunction with (\ref{eq:lem:disc}) shows that $\left|\Delta_{\mathbb{Z}[\theta]}\right|=n^{d-1}d^{d}$.
\end{proof}

Next, we apply Dedekind's criterion in our setting. It shows that if $p$ and $d$ are distinct primes and the Fermat quotient $\frac{p^{d-1}-1}{d}$ is not divisible by $d$ then  the field 
$\IQ(p^{1/d})$ is monogenic.
\begin{lemma}\label{lem:monogenic}
Let $p$ and $d$ be odd primes. Put $\theta=p^{1/d}$ and $K=\mathbb{Q}(\theta)$.
If $d^{2}\nmid p^{d}-p$ then $\Oseen_{K}=\mathbb{Z}[\theta]$. 
\end{lemma}

\begin{proof}
It is well-known that $\Delta_{\mathbb{Z}[\theta]}=\Delta_{K}[\Oseen_{K}:\mathbb{Z}[\theta]]^2$.
Hence, in order to rule out that $[\Oseen_{K}:\mathbb{Z}[\theta]]>1$, it
suffices by Lemma \ref{lem:disc} to check that $p,d\nmid[\Oseen_{K}:\mathbb{Z}[\theta]]$ as $x^d-p$ is irreducible
in $\IZ[x]$ by Eisenstein's criterion. 
We use
Lemma \ref{lem: dedekind's criterion} and its notation, with $M_{\theta,\IQ}$
as above. Next we study the two critical cases $q=p$ and $q=d$ separately.\\

\noindent Case $q=p$: Then $\overline{M_{\theta,\IQ}}=x^{d}=\varphi^{d}$ with
$\mu=x$, and $g=-1$. So $\varphi\nmid\overline{g}$. \\

\noindent Case $q=d$: Then $\overline{M_{\theta,\IQ}}=\overline{x^{d}-p}=(x-p)^{d}=\varphi^{d}$
with $\mu=x-p$, and 
\[
g=\frac{x^{d}-p-(x-p)^{d}}{d}=\frac{p^{d}-p}{d}-\sum_{i=1}^{d-1}
\frac{\begin{pmatrix}d\\
i
\end{pmatrix}}{d}x^{d-i}(-p)^{i}.
\]
So $\varphi\mid\overline{g}$ if and only if $\overline{g}(p)=0$.
Now $\overline{g}(p)=\frac{p^{d}-p}{d}-0=\frac{p^{d}-p}{d}\neq0$. 
\end{proof}

Next we want to show that for given $\t>1$ and each prime $d_i$ sufficiently large, there exists a prime $p_i$ such that the sequence $p_i^{1/d_i}$ converges
to $\t$ in a prescribed way (i.e. from above or from below). But we also want the fields $\IQ(p_i^{1/d_i})$ to be monogenic, and so we need that 
the Fermat quotients $\frac{p_i^{d_i-1}-1}{d_i}$ are not divisible by $d_i$. 
We use the observation $\frac{(d-1)^{d-1}-1}{d} \equiv  1 \pmod {d}$, for any odd prime $d$,
which follows from a straightforward computation 
(see the proof of the next lemma).
Thus, to achieve the required monogenicity it suffices to have
\begin{alignat}1\label{cond:Eisenstein}
p_i\equiv d_i-1 \pmod {d_i^2}. 
\end{alignat}

\begin{lemma}\label{lem:Eisenstein}
Let $d$ be an odd prime. We have
$\frac{(d-1)^{d-1}-1}{d} \equiv  1 \pmod {d}$.
\end{lemma}
\begin{proof}
First, note that 
$$
(d-1)^{d-1}= \sum_{0\leq j \leq d-1} \binom{d-1}{j} d^j (-1)^{d-1-j}
=(-1)^{d-1} 
+ (d-1)d (-1)^{d-2}
+d  \sum_{2\leq j \leq d-1} \binom{d-1}{j} d^{j-1} (-1)^{d-1-j}.
$$
Because $d$ is odd, the first two terms simplify to $ 1 - d(d-1) $. Hence,  
$$
\frac{(d-1)^{d-1}-1}{d}=
1 +d\bigg( -1+ \sum_{2\leq j \leq d-1} \binom{d-1}{j} d^{j-2} (-1)^{d-1-j}\bigg)
$$
which implies the claim.
\end{proof}

A straightforward application of the Siegel--Walfisz Theorem, 
see \cite[Cor. 5.29]{IW04}, guarantees the existence 
of the required primes $p_i$ in the right interval, and the prescribed residue class. 
\begin{lemma}\label{lem:SiegelWalfisz}
Fix $\t>1$. 
If $d$ is a sufficiently large prime  in terms of $\t$,  then there
exist a prime  $p$ in the interval $(\t^{d},2\t^{d})$ and one prime $p$ in the interval $(\t^{d}/2,\t^{d})$
such that in both cases $p\equiv d-1\pmod {d^2}$.
\end{lemma}

\begin{proof}
Let $\pi(x;q,a)$ denote the number of primes  $p\leq x$
that solve the congruence $p\equiv a\pmod q$. Let $\varphi$ denote
Euler's totient function, and let 
\[
\mathrm{Li}(x)=\int_{2}^{x}\frac{\mathrm{d}\tau}{\log\tau}.
\]
The Siegel--Walfisz theorem states for any $N>1$ and $q\geq 1$
we have
\[
\pi(x;q,a)=\frac{\mathrm{Li}(x)}{\varphi(q)}+O_{N}(x(\log x)^{-N})
\]
uniformly for all $1 \leq a \leq q$  with $(a,q)=1$ for any $x\geq 2$. 
We will only check that 
$$\pi(2\t^{d};d^2,d-1)-\pi(\t^{d};d^2,d-1)\geq 2.$$
Here the $2$ ensures that we have an element in the open interval $(\t^d,2\t^d)$.
The case $\pi(\t^{d};d^2,d-1)-\pi(\t^{d}/2;d^2,d-1)\geq 2$ is shown similarly.
Thus, specialising $x=\t^{d}$ and $N=4$, we infer that 
\begin{align*}
\pi(2x;d^{2},d-1)-\pi(x;d^{2},d-1) & =\frac{\mathrm{Li}(2x)}{\varphi(d^{2})}-\frac{\mathrm{Li}(x)}{\varphi(d^{2})}+O\left(\frac{x}{(\log x)^{4}}\right)\\
 & =\frac{1}{\varphi(d^{2})}\int_{x}^{2x}\frac{\mathrm{d}\tau}{\log\tau}
+O\left(\frac{x}{(\log x)^{4}}\right).
\end{align*}
Because $d=(\log x)/\log \t$, we have 
\[
\frac{1}{\varphi(d^{2})}\int_{x}^{2x}\frac{\mathrm{d}\tau}{\log\tau}>\frac{1}{\left(\frac{\log x}{\log \t}\right)^{2}}\frac{x}{\log(2x)}>(\log \t)^{2}\frac{x}{(1+\log x)^{3}}.
\]
Since $d$ is large (and hence $x$ as well), this implies $\pi(2x;d^{2},d-1)-\pi(x;d^{2},d-1)\geq 2$.
\end{proof}

We can now deduce:
\begin{lemma}\label{lem:compositeringofintegersapplied}
Let $(p_i,d_i)_i$ be a sequence where both components are prime, $\min\{p_{i+1}, d_{i+1}\}>\max\{p_{i}, d_{i}\}$ for all $i$, and such that  $\Oseen_{\IQ(\xi_i)}=\IZ[\xi_i]$ with $\xi_i=p_i^{1/d_i}$.
Then with $\displaystyle{K=\cup_{i\geq1} \IQ(\xi_1,\ldots,\xi_i)}$ and $\Oseen_i=\IZ[\xi_1,\ldots,\xi_i]$  we have
$\displaystyle{\Oseen_K=\bigcup_{i\geq1} \Oseen_i}$.
\end{lemma}
\begin{proof}
Set $K_i=\IQ(\xi_1,\ldots,\xi_i)$. The only primes that ramify in $K_{i-1}$ are $p_1,d_1,\ldots, p_{i-1},d_{i-1}$. Hence we have $(\Delta_{K_{i-1}}, \Delta_{\IQ(\xi_i)})=1$.
Moreover, $K_{i-1}$ and $\IQ(\xi_i)$ are linearly disjoint over $\IQ$ since their degrees $d_1\cdots d_{i-1}$ and $d_i$ are coprime.
We conclude from Lemma \ref{lem:compositeringofintegers} that $\Oseen_{K_i}=\Oseen_{K_{i-1}}\cdot \IZ[\xi_i]$, and hence by induction $\Oseen_{K_i}= \IZ[\xi_1,\ldots,\xi_i]=\Oseen_i$.
This proves that
$\Oseen_K=\cup_{i\geq1} \Oseen_i$.
\end{proof}

\section{Proof of Theorem \ref{thm: spectrum full for house function}}
We use the specific construction given in Corollary \ref{example:standardfields} with additional constraints on the primes $p_i$ and $d_i$
as required to ensure the specific properties.\\

(a): Let $(d_i)_i$ be a sequence of strictly increasing primes. Combining Lemma \ref{lem:monogenic}, the congruence condition (\ref{cond:Eisenstein}),
and Lemma \ref{lem:SiegelWalfisz}, we see that for each $i$ large enough there exists a prime $p_i\in (t^{d_i}/2,t^{d_i})$ such that $\Oseen_{\IQ(p_i^{1/d_i})}=\IZ[p_i^{1/d_i}]$.
We select a subsequence of $(p_i,d_i)_i$ to ensure $\min\{p_{i+1}, d_{i+1}\}>\max\{p_{i}, d_{i}\}$ for all $i$. 
Set $\xi_i=p_i^{1/d_i}$, $\Oseen_i=\IZ[\xi_1,\ldots,\xi_i]$, and  $\Oseen=\cup_i \Oseen_i$. 
The $\xi_i$ are pairwise distinct, and $2^{-1/d_i}t<\house{\xi_i}<t$. Thus $\#\{\alpha\in \Oseen; \housealp< t\}=\infty$, and it 
follows from  Corollary  \ref{example:standardfields} that $\N_{\house{\cdot}}(\Oseen)=t$.
Finally, by Lemma \ref{lem:compositeringofintegersapplied} we see that
$\Oseen_K=\Oseen$. This proves the existence of a field $K$ with $\N_{\house{\cdot}}(\Oseen_K)=t$ and satisfying (a).

(b): To construct a field $K$ with $\N_{\house{\cdot}}(\Oseen_K)=t$ and (b) we proceed in the same
manner but choose the primes $p_i\in (t^{d_i},2t^{d_i})$, so that  by Corollary  \ref{example:standardfields} 
$\eta(K_{i-1},\xi_i)=\house{\xi_i}\in (t,2^{1/d_i}t)$.
It follows from  Proposition \ref{prop:housebound} that
$\#\{\alpha\in \Oseen_K; \housealp\leq t\}<\infty$.

Now suppose $\epsilon>0$ and $M-t<(t-1)/2$, and suppose $\alpha \in \Oseen$ with $t+\epsilon<\housealp<M$.
Clearly, there are only finitely many such $\alpha$ in $\Oseen_0$, so we can assume $\alpha \notin \Oseen_0$.
Thus $\alpha \in \Oseen_i\backslash\Oseen_{i-1}$ for some $i\in \IN $. Note that $d_i=[\IQ(\xi_i):\IQ]$.
Hence, there exists  $P=\sum_k a_kx^k\in \Oseen_{i-1}[x]$ of degree $1\leq n< d_i$ such that $\alpha=P(\xi_i)$.
Since $M_{\xi_i,K_{i-1}}(x)=M_{\xi_i,\IQ}(x)=x^{d_i}-p_i\in \IQ[x]$ it follows that the $d_i$ distinct roots of  $\sigma(M_{\xi_i,K_{i-1}})(x)$
are perfectly equidistributed on the circle $|z|=\house{\xi_i}=p_i^{1/d_i}$, and therefore  $D\left(\frac{\vtau_\sigma(\xi_i)}{\|\vtau_\sigma(\xi_i)\|}\right)=0$ for all $\sigma\in \Hom(K_{i-1})$.
By Lemma \ref{lem:housebound} we conclude
that $\housealp\geq \house{\xi_i}^n>t^n$, and thus $t^n<t+(t-1)/2$. This forces $n<2$ and thus $n=1$. So $\alpha=a_1\xi_i+a_0$. 
Again, by Lemma \ref{lem:housebound} we get
that $\housealp\geq |\sigma(a_1)|\house{\xi_i}$ for all $\sigma\in \Hom(K_{i-1})$, and thus $\housealp\geq \house{a_1}\cdot \house{\xi_i}$.
It follows that $M>\housealp>\house{a_1}\cdot \house{\xi_i}>\house{a_1}\cdot t$, and thus $\house{a_1}<M/t<1+(t-1)/(2t)<1+(t-1)/2<t$.
Hence there are only finitely many possibilities for $a_1\in \Oseen$, in particular, there is an element with smallest house $>1$.
We can assume that $M/t$ is below the smallest house value $>1$ for such elements $a_1$. 
This forces $\house{a_1}=1$, which in turn
implies that all archimedian absolute values are equal to $1$ since $a_1$ is integral.

Next note that the sectors bounded by the $d_i$ rays starting at $0$ and joining the conjugates $\house{\xi_i}\zeta_{d_i}^j$ partition the complex plane.
Consider the sector that contains $\sigma(a_0)$ where $\sigma\in \Hom(K_{i-1})$ is such that $\house{a_0}=|\sigma(a_0)|$. 
Due to the coprimality of the degrees $d_i$ there exists an extension $\tau\in \Hom(K_{i})$ of $\sigma$ that sends $a_1\xi_i$ to a
conjugate that lies in the same sector as  $\sigma(a_0)$. It follows that 
$\housealp\geq|\tau(a_1\xi_i+a_0)|\geq |\tau(a_1\xi_i)|+|\cos(2\pi/d_i)\sigma(a_0)|$. 
We can assume that $d_i>6$ and thus $\cos(2\pi/d_i)\geq 1/2$. Hence we conclude 
$M>\housealp\geq \house{\xi_i}+\house{a_0}/2$.
Moreover, we can assume that $M<t+1/2$. As $\house{\xi_i}>t$ and $a_0$ is integral  we conclude that $a_0=0$. Thus 
$t+\epsilon<\housealp=\house{\xi_i}\leq 2^{1/d_i}t$, and thus $i$ is bounded in terms of $\epsilon$. Hence, there are only finitely many choices of $\alpha$, and 
this proves part (b).

(c): And finally, to construct a field $K$ with the third property we take a sequence $(T_i)_i$ of real numbers that converges to $t$ from above.
For each $T_i$ we construct a sequence  $(p_{ij}, d_{ij})_j$ with $p_{ij}\in (T_i^{d_{ij}},2T_i^{d_{ij}})$ and $\Oseen_{\IQ(p_{ij}^{1/d_{ij}})}=\IZ[p_{ij}^{1/d_{ij}}]$.
A slightly modified Cantor diagonalisation argument allows us to construct a sequence $(p_m,d_m)_m$ that contains infinitely many elements from each sequence
$(p_{ij}, d_{ij})_j$, and satisfies $\min\{p_{m+1}, d_{m+1}\}>\max\{p_{m}, d_{m}\}$ for all $m$. With $\xi_m=p_{m}^{1/{d_{m}}}$  we conclude from Lemma \ref{lem:compositeringofintegersapplied}  that $\Oseen_K=\Oseen=\IZ[\xi_m; m\in \IN]$. Now by Corollary  \ref{example:standardfields} 
we get $\N_{\house{\cdot}}(\Oseen_K)=t$, and by
Proposition \ref{prop:housebound} we see that 
$\housealp>t$ for all $\alpha\in \Oseen_K\backslash \IZ$. And, finally, since each $T_i$ is a limit point of $(\house{\xi_m})_m$ we see that 
$\#\{\alpha\in \Oseen_K; t+\epsilon\leq \housealp< M\}=\infty$ for all $M>t$ and all $\epsilon>0$ sufficiently small.
This finishes the proof of Theorem \ref{thm: spectrum full for house function}.

\section{Proofs of Theorem \ref{Prop2} and Theorem \ref{PropNBGexplicit}}\label{framework}

Theorem \ref{Prop2} and Theorem \ref{PropNBGexplicit} both use Lemma \ref{LemmaN}  applied to a common setting,
with the sets $A_{i}$ defined as follows. Let $(p_i)_i$ and $(d_i)_i$ be sequences of prime numbers.
Let $p_i^{1/d_i}$ be any choice of a $d_i$-th root of $p_i$.
We set $A_{i}=K_i$ where $K_0=\IQ$ and $K_{i+1}=K_i(p_i^{1/d_i})$, and we write $K=\cup_i K_i$.

\begin{lemma}\label{Lemdelta}
Let $\gamma\geq 0$, and recall that $h_{\gamma}(\alpha)=(\deg \alpha)^\gamma h(\alpha)$. Then $\N_{h_{\gamma}}(K_i)=\infty$ for all $i\in \IN_0$. Moreover, if $p_{i}\notin \{d_1,p_1,\ldots,d_{i-1},p_{i-1}\}$ then
$$\delta_{h_{\gamma}}(K_{i}\backslash K_{i-1})\geq d_i^{\gamma}\left(\frac{\log p_i}{2d_i}-\frac{\log d_i}{2(d_i-1)}\right).$$
\end{lemma}
\begin{proof}
First note that $\N_{h_{\gamma}}(K_i)=\infty$ by Northcott's Theorem. 
Next let us prove the inequality. Since only primes in $\{d_1,p_1,\ldots,d_{i-1},p_{i-1}\}$ can ramify in $K_{i-1}$ we conclude that $p_i$ is unramified in $K_{i-1}$, 
and hence $x^{d_i}-p_i$ is an Eisenstein polynomial in $\Oseen_{K_{i-1}}[x]$. Thus, $[K_i:K_{i-1}]=d_i$ is prime, and we conclude that
$K_{i-1}(\alpha)=K_i$ for any $\alpha\in K_{i}\backslash K_{i-1}$. An inequality of Silverman \cite[Theorem 2]{Silverman} (see also \cite[(5)]{WidmerPropN}) implies that
\begin{alignat*}1
h(\alpha)\geq 
\frac{\log (N_{K_{i-1}/\IQ}(D_{K_i/K_{i-1}}))}{2[K_{i-1}:\IQ]d_i(d_i-1)}-\frac{\log d_i}{2(d_i-1)},
\end{alignat*}
where $N_{K_{i-1}/\IQ}(\cdot)$ denotes the norm and $D_{K_i/K_{i-1}}$ denotes the relative discriminant.
A straightforward calculation shows (see \cite[Proof of Theorem 4]{WidmerPropN}) that $p_{i}^{[K_{i-1}:\IQ](d_i-1)}$ divides $N_{K_{i-1}/\IQ}(D_{K_i/K_{i-1}})$. Hence,
\begin{alignat*}1
h(\alpha)\geq \frac{\log p_i}{2d_i}-\frac{\log d_i}{2(d_i-1)}.
\end{alignat*}
Finally, we note that $\deg(\alpha)\geq [K_{i-1}(\alpha):K_{i-1}]=d_i$, and hence
\begin{alignat*}1
\delta_{h_{\gamma}}(K_{i}\backslash K_{i-1})\geq d_i^{\gamma}\left(\frac{\log p_i}{2d_i}-\frac{\log d_i}{2(d_i-1)}\right).
\end{alignat*}
\end{proof}

We are now ready for the proof of Theorem \ref{Prop2}:\\

\noindent{\it Proof of Theorem \ref{Prop2}:}\\
Let $(p_i)_i$ and $(d_i)_i$ be sequences of primes and the $p_i$ strictly increasing
such that $\log(p_i)/d_i$ converges to $2t$. It follows that $p_{i}\notin \{d_1,p_1,\ldots,d_{i-1},p_{i-1}\}$ for all sufficiently large $i$.
Set $K_0=\IQ$ and $K_{i}=K_{i-1}(p_{i}^{1/d_{i}})$ as at the beginning of this section. Applying Lemma \ref{Lemdelta} with $\gamma=0$ we conclude that $\liminf \delta_h(K_{i}\backslash K_{i-1})\geq t$, and hence
by Lemma \ref{LemmaN} we get $\N_h(K)\geq t$.
For the remaining inequality 
we note that $p_i^{1/d_i}$ are all integral (and all distinct for sufficiently large $i$) with height $\log(p_i)/d_i$ converging to $2t$, and thus we immediately get $\N_h(K)\leq \N_h(\Oseen_K)\leq 2t$.
\qed\\

For the proof of  Theorem \ref{PropNBGexplicit} we need the following two lemmas.

\begin{lemma}\label{LemNBG}
Let $0\leq \gamma\leq 1$, and let $(p_i)_i$ and $(d_i)_i$ be sequences of prime numbers such that ${d_i}^{\gamma-1}(\log p_i-\log d_i)\rightarrow \infty$ as $i\rightarrow \infty$, and
$p_i\notin \{d_1,p_1,\ldots,d_{i-1},p_{i-1}\}$ for all $i>i_0$.
Then $L=\IQ(p_i^{1/d_i}; i\in \IN)$ is  $\gamma$-Northcott. 
\end{lemma}
\begin{proof}
Since
\begin{alignat*}1
d_i^{\gamma}\left(\frac{\log p_i}{2d_i}-\frac{\log d_i}{2(d_i-1)}\right)=\frac{1}{2} \left(\frac{\log p_i-\log d_i}{d_i^{1-\gamma}}-\frac{\log d_i}{d_i^{2-\gamma}(1-1/d_i)}\right)
\geq \frac{1}{2} \left(\frac{\log p_i-\log d_i}{d_i^{1-\gamma}}-1\right),
\end{alignat*}
it follows from Lemma \ref{Lemdelta}
that for $i>i_0$ we have
\begin{alignat*}1
\delta_{h_\gamma}(K_{i}\backslash K_{i-1})\geq  \frac{1}{2}\left(\frac{\log p_i-\log d_i}{d_i^{1-\gamma}}-1\right)\rightarrow \infty.
\end{alignat*}
Hence, the claim follows from Lemma \ref{LemmaN}.
\end{proof}

\begin{lemma}\label{Lem2NBG}
Let $0\leq \gamma\leq 1$, $\epsilon>0$, and let $(p_i)_i$ and $(d_i)_i$ be sequences of prime numbers such that ${d_i}^{\gamma-1}(\log p_i-\log d_i)\rightarrow \infty$ 
and ${d_i}^{\gamma-\epsilon -1}\log p_i\rightarrow 0$ as $i\rightarrow \infty$,  and
$p_i\notin \{d_1,p_1,\ldots,d_{i-1},p_{i-1}\}$ for all $i>i_0$.
Then $L=\IQ(p_i^{1/d_i}; i\in \IN)$ is  $\gamma$-Northcott but not $(\gamma-\epsilon)$-Bogomolov. 
\end{lemma}
\begin{proof}
Since $h_{\gamma-\epsilon}(p_i^{1/d_i})=(\deg p_i^{1/d_i})^{\gamma-\epsilon} h(p_i^{1/d_i})={d_i}^{\gamma-\epsilon -1}\log p_i\rightarrow 0$
the claim follows immediately from Lemma \ref{LemNBG}.
\end{proof}

\noindent{\it Proof of Theorem \ref{PropNBGexplicit}:}\\
Theorem \ref{PropNBGexplicit} now follows easily from Lemma \ref{Lem2NBG} as  all sequences of primes $(d_i)_i$ with $d_{i+1}\geq 2d_{i}$, and  $(p_i)_i$ with $e^{d_i^{1-\gamma+\epsilon/2}}\leq p_i\leq 2e^{d_i^{1-\gamma+\epsilon/2}}$
satisfy the required conditions of Lemma \ref{Lem2NBG}. 
\qed

\bibliographystyle{amsplain}
\bibliography{literature}

\end{document}